\documentclass[10pt]{amsart}
\usepackage[margin=1.2in]{geometry}
\usepackage{color}
\newtheorem{theorem}{Theorem}[section]
\newtheorem{lemma}[theorem]{Lemma}
\newtheorem{proposition}{Proposition}[section]

\theoremstyle{definition}

\theoremstyle{remark}

\numberwithin{equation}{section}

\usepackage{graphicx}
\DeclareGraphicsExtensions{.pdf,.png,.jpg}
\begin{document}
	\title[Anderson-Witting model of the relativistic Boltzmann equation near equilibrium]{Anderson-Witting model of the relativistic Boltzmann equation near equilibrium}
	\author[B.-H. Hwang]{Byung-Hoon Hwang}
	\address{Department of Mathematics, Sungkyunkwan University, Suwon 440-746, Republic of Korea}
	\email{bhh0116@skku.edu}

	\author[S.-B. Yun]{Seok-Bae Yun}
	\address{Department of Mathematics, Sungkyunkwan University, Suwon 440-746, Republic of Korea}
	
	\email{sbyun01@skku.edu}
	
	
\begin{abstract}
%
%
Anderson-Witting model is a relaxational model equation of the relativistic Boltzmann equation, which sees a wide application in physics. In this paper,
we study the existence of classical solutions and its asymptotic behavior when the solution starts sufficiently close to  a global relativistic Maxwellian.
\end{abstract}
\maketitle
\section{Introduction}
In this paper, we consider the Cauchy problem for Anderson-Witting model \cite{AW}:
\begin{align}\label{AWRBGK1}
	\begin{split}
		\partial_t F+\hat{q}\cdot\nabla_x F&=\frac{U_\mu q^\mu}{q^0}(J(F)-F),\cr
		F_0(x,q)&=F(0,x,q).
	\end{split}
\end{align}
The momentum distribution function $F(x^\mu,q^\mu)$ represents the number density of relativistic particles at the phase point $(x^\mu ,q^\mu)~(\mu=0,1,2,3)$ in the Minkowski space where $x^\mu=(t,x)\in \mathbb{R}_+\times\mathbb{T}^3$ denotes the space-time coordinate and $q^\mu=(\sqrt{1+|q|^2},q)\in \mathbb{R}_+\times\mathbb{R}^3$ is the energy-momentum four-vector. The normalized momentum $\hat{q}$ is defined by $q/q^0$. The relativistic Maxwellian $J(F)$ is given by
$$
J(F)=\frac{n }{M(\beta )} e^{-\beta U^\mu q_\mu},
$$
where $M(\beta)$ denotes
$$
M(\beta)=\int_{\mathbb{R}^3} e^{-\beta\sqrt{1+|q|^2}}dq.
$$
Throughout this paper, we employ 
the signature of the metric $\eta^{\mu\nu}=\eta_{\mu\nu}=\text{diag}(1,-1,-,1,-1)$,
so that the Minkowski inner product $p^\mu q_\mu$ is given by
$$
p^\mu q_\mu =p^0 q^0-\sum_{i=1}^3 p^iq^i.
$$

To define the macroscopic fields, we consider the particle four-flow $N^\mu$ and energy-momentum tensor $T^{\mu\nu}$:
\begin{equation*}
N^\mu=\int_{\mathbb{R}^3}q^\mu F\frac{dq}{q^0},\qquad T^{\mu\nu}=\int_{\mathbb{R}^3}q^\mu q^\nu F\frac{dq}{q^0}.
\end{equation*}
Then, the particle density $n$ and the macroscopic velocity $U^{\mu}$ are defined by
\begin{equation}\label{n}	n^2=N^{\mu}N_{\mu}=\left(\int_{\mathbb{R}^3}F\,dq\right)^2-\sum_{i=1}^3\left(\int_{\mathbb{R}^3}Fq^i\,\frac{dq}{q^0}\right)^2,
\end{equation}
and 
$$
	 U^\mu=u^\mu+\frac{\mathbf{q}^\mu}{nh},
$$
where the Eckart four-velocity $u^\mu$, the heat flux $\mathbf{q}^\mu$, the enthalpy function $h$, the internal energy per particle $e$ and the pressure $p$ are defined as follows (We follow the Einstein summation convention):
\begin{align}\label{LLD1}
	\begin{split}
				u^\mu&=\frac{1}{n}N^\mu=\frac{1}{n}\int_{\mathbb{R}^3}Fq^\mu\,\frac{dq}{q^0},\cr
	\mathbf{q}^\mu&=\Delta^\mu_\gamma u_\nu T^{\nu\gamma}=\int_{\mathbb{R}^3}F (u^\nu q_\nu)q^\mu\frac{dq}{q^0}-u^\mu \int_{\mathbb{R}^3} F(u^\nu q_\nu)^2 \frac{dq}{q^0},\cr
			e&=\frac{1}{n}u^\mu u^\nu T_{\mu\nu}=\frac{1}{n} \int_{\mathbb{R}^3}F (u^\mu q_\mu)^2 \frac{dq}{q^0},\cr
		p&=-\frac{1}{3}\Delta^{\mu\nu}T_{\mu\nu}=\frac{1}{3}\left(\int_{\mathbb{R}^3}F (u^\mu q_\mu)^2 \frac{dq}{q^0}-\int_{\mathbb{R}^3}F  \frac{dq}{q^0} \right),\cr
			h&=e+\frac{p}{n}=\frac{1}{3n}\left(4\int_{\mathbb{R}^3}F (u^\mu q_\mu)^2 \frac{dq}{q^0}-\int_{\mathbb{R}^3}F  \frac{dq}{q^0} \right).
\end{split}\end{align}
The projection operators $\Delta^{\mu\nu}$ and $\Delta^\mu_\nu$ are defined by
\begin{equation*}
	\Delta^{\mu\nu}=\eta^{\mu\nu}-u^\mu u^\nu,\qquad \Delta^\mu_\nu= g^{\mu}_\nu - u^\mu u_\nu,
\end{equation*}
where $g^{\mu}_\nu$ denotes Kronecker delta. Note from \eqref{n} and $\eqref{LLD1}_1$ that
$$
u^\mu u_\mu=\frac{1}{n^2}N^\mu N_\mu=1,
$$
which implies $u^\mu$ takes the form of
$$
u^\mu=\left(\sqrt{1+|u|^2},u\right).
$$
Finally, the equilibrium temperature $1/\beta$ is determined through the following nonlinear relation:
\begin{align}\label{beta}
	\frac{K_1}{K_2}(\beta)+\frac{3}{\beta}&=e.
\end{align}
For later convenience, we denote
\begin{align}\label{e tilde}
\widetilde{e}(\beta)=\frac{K_1}{K_2}(\beta)+\frac{3}{\beta}.
\end{align}
The unique solvability of the nonlinear relation \eqref{beta} will be considered in section 3.

The r.h.s of \eqref{AWRBGK1} is called a relativistic relaxation operator that satisfies
\begin{align}\label{properties}
  U_\mu\int_{\mathbb{R}^3}\left(J(F)-F\right)q^\mu\left(\begin{matrix}
		1\\
		q^\nu
	\end{matrix}\right) \,\frac{dq}{q^0}=0,\qquad
  U_\mu\int_{\mathbb{R}^3}\left(J(F)-F\right)\ln Fq^\mu\,\frac{dq}{q^0}&\le 0.
	\end{align}
The cancellation property $\eqref{properties}_1$ gives the conservation laws of total mass, momentum and energy
\begin{equation*}
\frac{d}{dt} \int_{\mathbb{T}^3 }\int_{\mathbb{R}^3}F\left(\begin{matrix}
1\\
q^\nu
\end{matrix}\right)\,dqdx=0,
\end{equation*}
and $\eqref{properties}_2$ leads to the celebrated $H$-theorem:
$$
\frac{d}{dt}\int_{\mathbb{T}^3}\int_{\mathbb{R}^3}F \ln F\,dqdx\le 0.
$$

The classical BGK model \cite{BGK} is widely used in physics and engineering to understand the transport phenomena in a more simplified and numerically amenable manner. In the relativistic case,  there are two such relaxational models for the Boltzmann equation.
The first relaxation model was introduced by Marle \cite{Mar2,Mar3}.
Then Anderson and Witting \cite{AW} suggested another model that provides a better agreement with the relativistic Boltzmann equation in terms of the viscosity and the heat conductivity in the ultra-relativistic limit. The difference between them comes from the way in which the
macroscopic fields are represented. The Marle model uses the Eckart decomposition \cite{CK,E} to represent
the macroscopic fields, whereas the Landau-Lifshitz decomposition \cite{CK,LL} is employed for the
Anderson-Witting model. The Anderson-Witting model has been fruitfully applied to a wide range of physical problems in relativistic kinetic theory \cite{AW2,DHMNS, DHMNS2,FMRS,FRS,FRS2,Jaiswal1,Jaiswal2,JRS,KP,MBHS,MKSH,MNR,MW,TZP,ZTP}. However, the existence problems for Anderson-Witting model have never been addressed, which is the main motivation of the current work.

In this paper, we establish the global in time existence of unique smooth solution and its exponential decay to the equilibration when the initial data is sufficiently close to a global equilibrium state. For this we decompose the distribution function into a global equilibrium and the perturbation around it:
\begin{equation}\label{decomposition}
F=J^0+f\sqrt{J^0}
\end{equation}
where $J^0$ is a relativistic global Maxwellian defined by
\begin{equation}\label{GM}
J^0=\frac{1}{M(\beta_0)}e^{-\beta_0 q^0}.
\end{equation}
Inserting \eqref{decomposition}, the Anderson-Witting model \eqref{AWRBGK1} is rewritten as
\begin{align}\label{AWRBGK2}
	\begin{split}
		\partial_t f+\hat{q}\cdot\nabla_x f&=L(f)+\Gamma (f),\cr
		f_0(x,q)&=f(0,x,q),
\end{split}\end{align}
where the linearized relaxation operator $L$ and the nonlinear perturbation $\Gamma$ are given in Proposition \ref{lin3}. The initial perturbation $f_0$ is determined by $F_0=J^0+f_0\sqrt{J^0}.$

To state our main result, we introduce the following notations and definitions:
\begin{itemize}
	\item We define standard $L^2$ norm by
	\[\|f\|^{2}_{L^2_q}=\int_{\mathbb{R}^{3}}|f(q)|^2\,dq,\qquad\|f\|_{L^2_{x,q}}^{2}=\int_{\mathbb{T}^{3}}\int_{\mathbb{R}^{3}}|f(x,q)|^{2}\,dqdx.
	\]
	\item We define usual $L^2$ inner product by
	\[\langle f,g\rangle_q=\int_{\mathbb{R}^3}f(q)g(q)\, dq,\qquad\langle f,g\rangle_{x,q}=\int_{\mathbb{T}^3 }\int_{\mathbb{R}^{3}}f(x,q)g(x,q)\, dqdx .
	\]
	\item Multi-index $\alpha$ and $\beta$ are defined by
	$$
	\alpha=[\alpha_{0},\alpha_{1},\alpha_{2},\alpha_{3}],\ \beta=[\beta_{1},\beta_{2},\beta_{3}]
	$$
	and
	$$
	\partial^{\alpha}_{\beta}=\partial_{t}^{\alpha_{0}}\partial_{x^{1}}^{\alpha_{1}}\partial_{x^{2}}^{\alpha_{2}}\partial_{x^{3}}^{\alpha_{3}}\partial_{q^{1}}^{\beta_{1}}\partial_{q^{2}}^{\beta_{2}}\partial_{q^{3}}^{\beta_{3}}.
	$$
	\item We use $\mathbb{P}(x,y,\cdots)$ to denote a homogeneous generic polynomial:
	$$
	\mathbb{P}(x_1,x_2,\cdots, x_n)=\sum_{i}C_{i}x_1^{a_{i,1}}x_2^{a_{i,2}}\cdots x_n^{a_{i,n}}.
	$$
where $a_{i,j}$ are sequences of nonnegative integers. By homogeneous polynomial we mean that $\mathbb{P}(0)=0$.
	\item We define the energy functional $E$ by
	\begin{align*}
		E(f)(t)&=\sum_{|\alpha|+|\beta|\le N}\|\partial^{\alpha}_{\beta}f\|^{2}_{L^2_{x,q}}.
	\end{align*}
When there's no risk of confusion, we  use $E(f)(t)$ and $E(t)$ interchangeably for brevity.
\end{itemize}
The main result of this paper is as follows:
\begin{theorem}\label{main2}
	Let $N\geq 4$. Assume $F_{0}=J^0+\sqrt{J^0}f_{0}\geq0$ and suppose $F_0$ and $J^0$ have the same total mass, momentum and energy:
	\begin{align}\label{conserv1}
		\int_{\mathbb{T}^3 }\int_{\mathbb{R}^{3}}f_0\sqrt{J^0}\left(\begin{matrix}
			1\\
			q\\
			q^0
		\end{matrix}\right)\,dqdx=0.
	\end{align}
Then, if $ E(f_0)$ is sufficiently small, there exists a unique global in time solution to \eqref{AWRBGK1} satisfying
	\begin{enumerate}
		\item The momentum distribution function is non-negative:
		$$F(t,x,q)=J^0+\sqrt{J^0}f(t,x,q)\geq0.$$
	 \item The energy functional is uniformly bounded:
		\begin{align*}
			E(f)(t)+\int_0^t E(f)(s)\,ds\leq CE\big(f_0\big).
		\end{align*}
		\item The energy functional decays exponentially fast:
		\begin{align*}
			E(f)(t)\leq Ce^{-C^\prime t}E(f_0).
		\end{align*}
		
	\end{enumerate}	
\end{theorem}

The mathematical study of relativistic BGK models has just started and the literature is limited. For the Marle model, Bellouquid et al. \cite{BCNS} carried out an initial research in 2012 where the determination of equilibrium variables, asymptotic limits and linearized solution were considered.  The global existence of mild solution and its asymptotic behavior in the periodic domain is studied in \cite{BNU}. In 2018, the authors \cite{HY} studied the stationary problem in a slab. 
To the best of our knowledge, no existence results were reported for Anderson-Witting model (\ref{AWRBGK1}) so far.

The situation is far better in the case of the relativistic Boltzmann equation. The local existence and linearized solution was studied in \cite{B,D,DE}. We refer to \cite{GS1,GS2,Guo Strain Momentum,Strain1,Strain Zhu} for the global existence and asymptotic behavior in the near-equilibrium regime, and \cite{Dud3,Jiang1,Jiang2} for the existence of the general large data. For the study on the spatially homogeneous case, we refer to \cite{LR,Strain Yun}. The propagation of the uniform upper bound was recently established in \cite{JSY}. Results on the regularizing effect of the gain term can be found in \cite{A,JY,W}.
We refer to \cite{Cal,Strain2} for the Newtonian limit and \cite{SS} for the hydrodynamic limit.\newline


This paper is organized as follows: In Section 2, various useful technical lemmas are presented.
In Section 3, we study the linearization of the Anderson-Witting relaxation operator. In Section 4, we provide the estimates for the macroscopic fields and nonlinear perturbation.
The proof of the main theorem is given in Section 5.
%
%
%
%
%

	%
	%
	%
	%
	%

%
%
%
%
%
\section{Preliminaries}
We record useful results for the Lorentz transformation, and the modified Bessel function of the second kind:
\begin{align*}
	K_{i}(\beta)&=\int_{0}^{\infty}\cosh(ir)\exp\left\{-\beta \cosh(r)\right\}dr.
\end{align*}
Some of them can be found in \cite{BCNS} with or without proof. Even in the former case, we provide the proof for reader's convenience.
The following identities obtained from the use of the change of variable $y= \sinh r$ are frequently used throughout this section:
\begin{align}\label{CV}
	\begin{split}
		K_{0}(\beta)&=\int_{0}^{\infty}\frac{1}{\sqrt{1+y^{2}}}\exp\left\{-\beta
		\sqrt{1+y^{2}}\right\}\,dy,\cr
		K_{1}(\beta)&=\int_{0}^{\infty}\exp\left\{-\beta \sqrt{1+y^{2}}\right\}\,dy, \cr
		K_{2}(\beta)&=\int_{0}^{\infty}\frac{2y^{2}+1}{\sqrt{1+y^{2}}}\exp\left\{-\beta
		\sqrt{1+y^{2}}\right\}\,dy.
	\end{split}
\end{align}
The following lemma enables us to compute various quantities in the local rest frame.
\begin{lemma}\cite{Strain2}\label{rest frame}
	For $U^\mu=(\sqrt{1+|U|^2},U)$, define $\Lambda$ by
	\begin{align*}
		\Lambda=
		\begin{bmatrix}
			U^0 & -U^1 & -U^2 & -U^3 \cr
			-U^1&  1+(U^0-1)\frac{(U^1)^2}{|U|^2}&(U^0-1)\frac{U^1U^2}{|U|^2}  &(U^0-1)\frac{U^1U^3}{|U|^2}  \cr
			-U^2& (U^0-1)\frac{U^1U^2}{|U|^2} &  1+(U^0-1)\frac{(U^2)^2 }{|U|^2}&(U^0-1)\frac{U^2U^3}{|U|^2}  \cr
			-U^3&  (U^0-1)\frac{U^1U^3}{|U|^2}& (U^0-1)\frac{U^2U^3}{|U|^2} &  1+(U^0-1)\frac{(U^3)^2}{|U|^2}
		\end{bmatrix}.
	\end{align*}
Then, $\Lambda$ transforms $U^\mu$ into the local rest frame:
 $$\Lambda U^{\mu}=(1,0,0,0).$$

\end{lemma}
\begin{proof}
The proof that $\Lambda$ is the Lorentz transformation can be found in \cite{Strain2}. The identity $\Lambda U=(1,0,0,0)$ can be verified by an explicit computation:
	\begin{align*}
		\Lambda U^\mu&=
		\begin{bmatrix}
			(U^0)^2-(U^1)^2-(U^2)^2-(U^3)^2\cr
			-U^0U^1+U^1+\frac{(U^0-1)U^1}{|U|^2}|U|^2\cr
			-U^0U^2+U^2+\frac{(U^0-1)U^2}{|U|^2}|U|^2\cr
			-U^0U^3+U^3+\frac{(U^0-1)U^3}{|U|^2}|U|^2
		\end{bmatrix}=\begin{bmatrix}
			\big(\sqrt{1+|U|^2}\big)^2-|U|^2\cr
			-U^0U^1+U^1+(U^0-1)U^1 \cr
			-U^0U^2+U^2+(U^0-1)U^2 \cr
			-U^0U^3+U^3+(U^0-1)U^3
		\end{bmatrix}=\begin{bmatrix}
			1\cr
			0 \cr
			0 \cr
			0
		\end{bmatrix}.
	\end{align*}
	
\end{proof}
\begin{lemma}\cite{BCNS}\label{KM}
$K_i(\beta)$ $(i=1,2)$ are related to $M(\beta)$ through the following identity:
	\begin{equation*}
		\frac{K_{1}(\beta)}{K_{2}(\beta)}+\frac{3}{\beta}=-\frac{M^{\prime}(\beta)}{M(\beta)}.
	\end{equation*}
\end{lemma}
\begin{proof}
Recall that $M(\beta)$ takes the form of
$$
			M(\beta)=\int_{\mathbb R^{3}}\exp
\left\{-\beta\sqrt{1+|q|^{2}}\right\}dq.
$$
Using the spherical coordinates and integration by parts, we have
	\begin{align}\label{K_2}
		\begin{split}
			M(\beta)	&=\int_{\mathbb R^{3}}\exp
			\left\{-\beta\sqrt{1+|q|^{2}}\right\}\,dq\cr
				&=4\pi\int_{0}^{\infty}y^{2}\exp \left\{-\beta\sqrt{1+y^{2}}\right\}\,dy\cr
				&=-\frac{4\pi}{\beta}\int_{0}^{\infty}y\sqrt{1+y^{2}}\frac{d}{dy}\left\{\exp
			\left\{-\beta\sqrt{1+y^{2}}\right\}\right\}\,dy\cr
				&=\frac{4\pi}{\beta}\int_{0}^{\infty}\frac{2y^{2}+1}{\sqrt{1+y^{2}}}\exp\left\{-\beta
			\sqrt{1+y^{2}}\right\}\,dy\cr
			&=\frac{4\pi}{\beta}K_{2}(\beta).
		\end{split}
	\end{align}
	On the other hand, differentiating $K_2(\beta)$ leads to
	\begin{align}\label{dK_2}
		\begin{split}
		\frac{d}{d\beta}\left\{K_{2}(\beta)\right\}&=-\int_{0}^{\infty}(2y^{2}+1)\exp\left\{-\beta \sqrt{1+y^{2}}\right\}\,dy\cr
		&=\frac{2}{\beta}\int_{0}^{\infty}y\sqrt{1+y^{2}}\frac{d}{dy}\left\{\exp\left\{-\beta \sqrt{1+y^{2}}\right\}\right\}\,dy-K_{1}(\beta)\cr
		&=-\frac{2}{\beta}\int_{0}^{\infty}\frac{2y^{2}+1}{\sqrt{1+y^{2}}}\exp\left\{-\beta \sqrt{1+y^{2}}\right\}\,dy-K_{1}(\beta)\cr
		&=-\frac{2}{\beta}K_{2}(\beta)-K_{1}(\beta),
\end{split}	\end{align}
	which, together with \eqref{K_2}, gives
	\begin{align*}
		\frac{d}{d\beta}\left\{M(\beta)\right\}&=\frac{d}{d\beta}\left\{\frac{4\pi}{\beta}K_{2}(\beta)\right\}\cr
		&=-\frac{4\pi}{\beta^{2}}K_{2}(\beta)+\frac{4\pi}{\beta}\frac{d}{d\beta}\left\{K_{2}(\beta)\right\}\cr
		&=-\frac{3}{\beta}M(\beta)-\frac{4\pi}{\beta}K_{1}(\beta).
 	\end{align*}
	Dividing the last identity by $M(\beta)$ gives the desired result.
\end{proof}
\begin{lemma}\cite{BCNS}\label{relationB}
$K_i(\beta)$ $(i=0,1,2)$ satisfy the following relation:
	$$
	K_2(\beta)=\frac{2}{\beta}K_1(\beta)+K_0(\beta).
	$$
\end{lemma}
\begin{proof}
	
	It is straightforward from \eqref{CV} that
\begin{align*}	
			K_{2}(\beta)&=\int_{0}^{\infty}\frac{2y^{2}+1}{\sqrt{1+y^{2}}}\exp\left\{-\beta \sqrt{1+y^{2}}\right\}\,dy\cr
			&=\int_{0}^{\infty}\frac{2y^{2}}{\sqrt{1+y^{2}}}\exp\left\{-\beta \sqrt{1+y^{2}}\right\}\,dy+K_{0}(\beta).
			\end{align*}
Then, from simple integration by parts, the first term on r.h.s can be expressed as
\begin{align*}
	\int_{0}^{\infty}\frac{2y^{2}}{\sqrt{1+y^{2}}}\exp\left\{-\beta \sqrt{1+y^{2}}\right\}\,dy			&=\int_{0}^{\infty}-\frac{2y}{\beta}\frac{d}{dy}\left\{\exp\left\{-\beta \sqrt{1+y^{2}}\right\}\right\}\,dy \cr
	&=\frac{2}{\beta}\int_{0}^{\infty}\exp\left\{-\beta \sqrt{1+y^{2}}\right\}\,dy \cr
	&=\frac{2}{\beta}K_{1}(\beta),
\end{align*}
which gives the desired result.
	
\end{proof}


\begin{lemma}\cite{BCNS}\label{Bessel2}
The following identity holds for $K_i(\beta)$ $(i=1,2)$
	$$
	\left(\frac{K_{1}(\beta)}{K_{2}(\beta)} \right)^{\prime}=\frac{3}{\beta}\frac{K_{1}(\beta)}{K_{2}(\beta)}+\left(\frac{K_{1}(\beta)}{K_{2}(\beta)} \right)^{2}-1.
	$$
\end{lemma}
\begin{proof}
	Differentiating $K_1(\beta)$ with respect to $\beta$, we have from \eqref{CV} that
	\begin{align}\label{K_1}
		\begin{split}
			\frac{d}{d\beta}	\left\{K_{1}(\beta)\right\}&=-\int_{0}^{\infty}\sqrt{1+y^{2}}\exp\left\{-\beta \sqrt{1+y^{2}}\right\}\,dy\cr
		&=-\frac{1}{2}\int_{0}^{\infty}\left(\frac{2y^2+1}{\sqrt{1+y^{2}}}+\frac{1}{\sqrt{1+y^{2}}}\right)\exp\left\{-\beta \sqrt{1+y^{2}}\right\}\,dy\cr
		&=-\frac{1}{2}\left(K_{2}(\beta)+K_{0}(\beta)\right).
\end{split}	
\end{align}
	We then recall from Lemma \ref{relationB} that
$$
	K_2(\beta)=\frac{2}{\beta}K_1(\beta)+K_0(\beta)
$$
to express the last identity in r.h.s of \eqref{K_1} as
$$
-\frac{1}{2}\left(K_{2}(\beta)+K_{0}(\beta)\right)=-\frac{1}{\beta}K_{1}(\beta)-K_{0}(\beta)
$$
so that
\begin{equation}\label{K_1-1}
\frac{d}{d\beta}	\left\{K_{1}(\beta)\right\}=-\frac{1}{\beta}K_{1}(\beta)-K_{0}(\beta).
\end{equation}
In the same manner, we have from \eqref{dK_2} that
	\begin{align}\label{K_1-2}
		\begin{split}
		\frac{d}{d\beta}	\left\{K_{2}(\beta)\right\}
		&=-\frac{2}{\beta}K_{2}(\beta)-K_{1}(\beta)\cr
		&=\left(-\frac{4}{\beta^{2}}-1\right)K_{1}(\beta)-\frac{2}{\beta}K_{0}(\beta).
	\end{split}\end{align}
Combining \eqref{K_1-1} and \eqref{K_1-2}, we have
	\begin{align*}
		&	\frac{d}{d\beta}	\left\{K_{1}(\beta)\right\}K_{2}(\beta)-K_{1}(\beta)	\frac{d}{d\beta}	\left\{K_{2}(\beta)\right\}\cr
		&=	-\frac{1}{\beta}K_{1}(\beta)K_{0}(\beta)-\left(K_{0}(\beta)\right)^{2}+\left(\frac{2}{\beta^{2}}+1\right)\left(K_{1}(\beta)\right)^{2}\cr
		&=-\frac{1}{\beta}K_{1}(\beta)\left(K_2(\beta)-\frac{2}{\beta}K_1(\beta)\right)-\left(K_2(\beta)-\frac{2}{\beta}K_1(\beta)\right)^{2}+\left(\frac{2}{\beta^{2}}+1\right)\left(K_{1}(\beta)\right)^{2}\cr
		&=\frac{3}{\beta}K_{1}(\beta)K_{2}(\beta)+(K_{1}(\beta))^{2}-(K_{2}(\beta))^{2}.
	\end{align*}
We divide the both sides by $(K_2(\beta))^2$ to obtain the desired result.
\end{proof}

	\section{Linearization}	
In this section, we study the linearization of Anderson-Witting model \eqref{AWRBGK1} around the global relativistic Maxwellian \eqref{GM}.
\subsection{Unique determination of $\beta$} Before we linearize the Anderson-Witting model \eqref{AWRBGK1}, we need to resolve the question raised in introduction, namely, that the nonlinear relation \eqref{beta} uniquely determine $\beta$.
First, we need to prove the following monotonicity result.
\begin{lemma}\label{beta decreasing}
The function $\widetilde{e}(\beta)$ defined in (\ref{e tilde}), satisfies the following properties:
\begin{enumerate}
\item  $\widetilde{e}(\beta)$ is strictly decreasing on $0<\beta<\infty$.
\item  $\widetilde{e}(\beta)>1$ on $0<\beta<\infty$.
\end{enumerate}
\end{lemma}
	\begin{proof}
\noindent$(1)$ Strict monotonicity: For $\beta\in (0,1)$, it follows from Lemma \ref{Bessel2} that
		\begin{align*}
			\biggl(\frac{K_{1}(\beta)}{K_{2}(\beta)}  +\frac{3}{\beta}\biggl)^{\prime}&=\frac{3}{\beta}\frac{K_1(\beta)}{K_2(\beta)} +\left(\frac{K_1(\beta)}{K_2(\beta)} \right)^2-1-\frac{3}{\beta^2}\cr
			&\le \frac{3}{\beta}\left(1-\frac{1}{\beta}\right)\cr
			&< 0.
		\end{align*}
		Here we used the fact that
		$$
		0<\frac{K_1(\beta)}{K_2(\beta)}=\frac{\int_{0}^{\infty}\exp\left\{-\beta \sqrt{1+y^{2}}\right\}\,dy}{\int_{0}^{\infty}\frac{2y^{2}+1}{\sqrt{1+y^{2}}}\exp\left\{-\beta
			\sqrt{1+y^{2}}\right\}\,dy}<1.
		$$
		 For $\beta\in [1,\infty)$, we use the change of variable $z=\sinh (r/2)$ to see that
		$$
		K_0(\beta)+K_1(\beta)=\int_{0}^{\infty}\left(1+\cosh(r)\right)e^{-\beta
			\cosh(r)}dr=4e^{-\beta}\int_0^\infty \frac{1+z^2}{\sqrt{1+z^2}}e^{-2\beta z^2}\, dz.
		$$
		This together with
		$$
		\frac{1}{\sqrt{1+z^2}}\le 1-\frac{z^2}{2}+\frac{3}{8}z^4,\qquad \int_{0}^{\infty}e^{-2\beta z^{2}}\,dz=\sqrt{\frac{\pi}{8\beta}},
		$$
		leads to
		\begin{align*}
			K_0(\beta)+K_1(\beta)&=4e^{-\beta}\int_0^\infty \frac{1+z^2}{\sqrt{1+z^2}}e^{-2\beta z^2}\, dz\cr
			&\le \frac{1}{2}e^{-\beta}\int_0^\infty \left(8+4z^2-z^4+3z^6 \right)e^{-2\beta z^2}\,dz\cr
			&=\frac{1}{2}e^{-\beta}\left(8+\frac{1}{\beta}-\frac{3}{16\beta^2}+\frac{45}{64\beta^3} \right)\sqrt{\frac{\pi}{8\beta}}.
		\end{align*}
		On the other hand, it follows from
		$$
		\frac{1}{\sqrt{1+z^2}}\ge 1-\frac{z^2}{2}
		$$
		that
		\begin{align*}
			K_0(\beta)&=2e^{-\beta}\int_0^\infty \frac{1}{\sqrt{1+z^2}}e^{-2\beta z^2}\,dz\cr
			&\ge 2e^{-\beta}\int_0^\infty \left( 1-\frac{z^2}{2} \right)e^{-2\beta z^2}\,dz\cr
			&=2e^{-\beta}\left(1-\frac{1}{8\beta} \right)\sqrt{\frac{\pi}{8\beta}}.
		\end{align*}
		Combining these estimates, we have
		$$
		\frac{K_0(\beta)+K_1(\beta)}{K_0(\beta)}\le \frac{512\beta^3+64\beta^2-12\beta+45}{256\beta^3-32\beta^2},
		$$
		which implies
		$$
		\frac{K_1(\beta)}{K_0(\beta)}\le \frac{256\beta^3+96\beta^2-12\beta+45}{256\beta^3-32\beta^2}.
		$$
		Also, using Lemma \ref{relationB} gives
		$$
		\frac{K_2(\beta)}{K_1(\beta)}\ge \frac{2}{\beta}+\frac{256\beta^3-32\beta^2}{256\beta^3+96\beta^2-12\beta+45}=\frac{256\beta^4+480\beta^3+192\beta^2-24\beta+90}{256\beta^4+96\beta^3-12\beta^2+45\beta}.
		$$
		Therefore we have from Lemma \ref{Bessel2} that
		\begin{align*}
			\biggl(\frac{K_{1}(\beta)}{K_{2}(\beta)} +\frac{3}{\beta}\biggl)^{\prime}&=\frac{3}{\beta}\frac{K_1}{K_2}(\beta)+\left(\frac{K_1}{K_2}(\beta)\right)^2-1-\frac{3}{\beta^2}\cr
			&\le \frac{3}{\beta}\frac{256\beta^4+96\beta^3-12\beta^2+45\beta}{256\beta^4+480\beta^3+192\beta^2-24\beta+90}\cr
			&+\left(\frac{256\beta^4+96\beta^3-12\beta^2+45\beta}{256\beta^4+480\beta^3+192\beta^2-24\beta+90}\right)^2-1-\frac{3}{\beta^2}\cr
			&< 0,
		\end{align*}
		which completes the proof of $(1)$.\newline

\noindent $(2)$ If $\beta\in (0,2)$, the desired result follows easily from the positivity of $K_i$:
		$$
		\widetilde{e}(\beta)=\frac{K_1(\beta)}{K_2(\beta)}+\frac{3}{\beta}\geq\frac{3}{\beta}>1.
		$$
		For the case $\beta\in [2,\infty)$ We recall the following inequality from \cite[Appendix]{BCNS}:
		\begin{align*}
			\frac{K_1(\beta)}{K_2(\beta)} &\ge \frac{128\beta^3+48\beta^2-33\beta}{128\beta^3+240\beta^2+105\beta-66}\cr
			&\ge 1+\frac{-192\beta^2-138\beta+66}{128\beta^3+240\beta^2+105\beta-66},
		\end{align*}
		which yields
		\begin{align*}
			  \frac{K_1(\beta)}{K_2(\beta)} +\frac{3}{\beta}&\ge 1+\frac{192\beta^3+582\beta^2+381\beta-198}{128\beta^4+240\beta^3+105\beta^2-66\beta}\cr
			&> 1.
		\end{align*}
	\end{proof}	
In the following proposition, we show that (\ref{beta}) admits a unique solution, at least, when the solution is sufficiently close to equilibrium.
\begin{proposition}\label{beta decreasing2}
Suppose $E(f)(t)$ is sufficiently small. Then \eqref{beta} uniquely determines $\beta$. Therefore we can write
$$
\beta=(\widetilde{e})^{-1}(e).
$$
\end{proposition}
\begin{proof}
Recall from \eqref{beta} that $\beta$ is determined by the nonlinear relation:
$$
	\frac{K_1(\beta)}{K_2(\beta)} +\frac{3}{\beta}=e,
$$
where $e$ is given in \eqref{LLD1}. First we denote $e_0$ by
$$
e_0=\frac{K_1(\beta_0)}{K_2(\beta_0)}+\frac{3}{\beta_0},
$$
and observe  from Lemma \ref{beta decreasing} that $e_0>1$. Then, Lemma $\ref{lem2}$ $(1)$ implies that
\begin{equation*}
e\ge e_0-\sqrt{E(f)(t)}>1,
\end{equation*}
when $E(f)(t)$ is sufficiently small.
We mention that the estimate in Lemma \ref{lem2} depends only on the moment estimate of $f$, and it's free from circular argument.
Therefore we can conclude that $e$ lies in the range of $\widetilde{e}(\beta)$. Then the strict monotonicity of $\widetilde{e}(\beta)$ which is proved in Lemma \ref{beta decreasing} gives the desired result.
 \end{proof}

\subsection{Linearization of Anderson-Witting model}
We start with the linearization of the relativistic Maxwellian $J(F)$.
We first provide the following lemma which is frequently used throughout this paper.
\begin{lemma}\label{J^0}
	$J^0$ satisfies
	$$
	\int_{\mathbb{R}^3}\left(1,~q,~q^0,~(q^i)^2,~qq^0,~(q^0)^2\right) J^0\,\frac{dq}{q^0}=\left(\frac{K_1(\beta_0)}{K_2(\beta_0)},~0~,1~,\frac{1}{\beta_0},~0,~\frac{K_1(\beta_0)}{K_2(\beta_0)}+\frac{3}{\beta_0}\right).
	$$
\end{lemma}
\begin{proof}
We have from the spherical coordinates and integration by parts
\begin{align*}
	\int_{\mathbb{R}^3}J^0\,\frac{dq}{q^0}&=\frac{4\pi}{M(\beta_0)}\int_0^\infty  \frac{y^2}{\sqrt{1+y^2}}e^{-\beta_0\sqrt{1+y^2}}\,dy \cr
	&=	\frac{4\pi}{\beta_0M(\beta_0)}\int_0^\infty e^{-\beta_0\sqrt{1+y^2}}\,dy,
\end{align*}
which, together with $\eqref{CV}_2$ and \eqref{K_2}, gives
$$
	\int_{\mathbb{R}^3}J^0\,\frac{dq}{q^0}=\frac{K_1(\beta_0)}{K_2(\beta_0)}.
$$
In the same manner, we find
\begin{align*}
	\int_{\mathbb{R}^3}q^0J^0\,dq&=	\frac{1}{M(\beta_0)}\int_{\mathbb{R}^3}\sqrt{1+|q|^2}e^{-\beta_0 \sqrt{1+|q|^2}}\,dq\cr
&=\frac{4\pi}{\beta_0 M(\beta_0)}	\left(\int_0^\infty  e^{-\beta_0 \sqrt{1+y^2}}\,dy+\frac{3}{\beta_0}\int_0^\infty \frac{2y^2+1}{\sqrt{1+y^2}}e^{-\beta_0 \sqrt{1+y^2}}\,dy\right)\cr
	&=\frac{K_1(\beta_0)}{K_2(\beta_0)}+\frac{3}{\beta_0}.
\end{align*}
On the other hand, by the spherical symmetry, we see that
\begin{align*}
	\int_{\mathbb{R}^3}(q^i)^2J^0\,\frac{dq}{q^0}&=	\frac{1}{M(\beta_0)}\int_{\mathbb{R}^3}\frac{1}{3}\frac{|q|^2}{\sqrt{1+|q|^2}}e^{-\beta_0\sqrt{1+|q|^2}}\,dq,
\end{align*}
so that
\begin{align*}
	\int_{\mathbb{R}^3}(q^i)^2J^0\,\frac{dq}{q^0}&=	\frac{4\pi}{3M(\beta_0)}\int_0^\infty\frac{y^4}{\sqrt{1+y^2}}e^{-\beta_0\sqrt{1+y^2}}\,dy\cr
	&=\frac{4\pi}{\beta_0M(\beta_0)}\int_0^\infty y^2e^{-\beta_0\sqrt{1+y^2}}\,dy\cr
	&=\frac{1}{\beta_0}.
\end{align*}
Finally, it is straightforward that
	$$
	\int_{\mathbb{R}^3}\left(q,q^0,qq^0\right) J^0\,\frac{dq}{q^0}=\left(0,1,0\right).
	$$

\end{proof}

We now linearize the relativistic Maxwellian $J(F)$.
 \begin{lemma}\label{lin2}
 Suppose $E(f)(t)$ is sufficiently small. Then, for $F=J^0+ f\sqrt{J^0}$, we have
 	$$
 	\frac{J(F)-J^0}{\sqrt{J^0}}=P(f)+\sum_{i=1}^5\Gamma_i (f).
 	$$
 $\bullet$	The projection operator $P$ is given by
 	\begin{align*}
 		P(f)&=\left(\int_{\mathbb{R}^3} f\sqrt{J^0}\,dq\right)\sqrt{J^0}+\frac{\beta_0}{\widetilde{h}(\beta_0)}\left(\int_{\mathbb{R}^3} qf\sqrt{J^0}\,dq\right)\cdot q\sqrt{J^0}\cr
 		 		&-\frac{1}{\left\{\widetilde{e}\right\}^{\prime}(\beta_0)}\left(\int_{\mathbb{R}^3} \left(q^0-e_0\right)f\sqrt{J^0}\,dq\right)\left(q^0-e_0\right)\sqrt{J^0},
 		\end{align*}
where $\widetilde{h}(\beta)$ denotes
$$
\widetilde{h}(\beta)=\frac{K_1(\beta)}{K_2(\beta)} +\frac{4}{\beta}.
$$
$\bullet$ The nonlinear perturbations $\Gamma_i(f)$ $(i=1,\cdots, 5)$ are given by
 	\begin{align*}
\Gamma_1 (f)&=\left(\frac{\Psi_1}{2}-\frac{\Psi^2}{2(2+\Psi+2\sqrt{1+\Psi})}\right)\sqrt{J^0},\cr
 	 \Gamma_2 (f)&=\biggl\{\left(\frac{\Psi_1}{2}+\frac{\Psi^3-3\Psi^2}{2(2+\Psi-\Psi^2+2\sqrt{1+\Psi})}\right) \left(e_0 +\int_{\mathbb{R}^3} q^0f\sqrt{J^0}\,dq   +\int_{\mathbb{R}^3} \left(2q^0\Phi+\Phi^2\right)F\,\frac{dq}{q^0}\right)\cr
 	 &+\int_{\mathbb{R}^3} f\sqrt{J^0}\,dq \int_{\mathbb{R}^3} q^0f\sqrt{J^0}\,dq+ \left(1+\int_{\mathbb{R}^3} f\sqrt{J^0}\,dq\right) \int_{\mathbb{R}^3} \left(2q^0\Phi+\Phi^2\right)F\,\frac{dq}{q^0}\biggl\}\cr
 	 &\quad\times \frac{1}{\left\{\widetilde{e}\right\}^{\prime}(\beta_0)}\left(q^0-e_0\right)\sqrt{J^0},\cr
	\Gamma_{3}(f)&=-\beta_0\left\{\frac{\Psi}{2}+\frac{\Psi^3-3\Psi^2}{2(2+\Psi-\Psi^2+2\sqrt{1+\Psi})}\right\}\int_{\mathbb{R}^3} q f\sqrt{J^0}\frac{dq}{q^0}\cdot q\sqrt{J^0}+\frac{\beta_0}{\widetilde{h}(\beta_0)}\Gamma_{3}^*(f)\cdot q\sqrt{J^0}\cr
 	&-\beta_0\left\{\frac{4\int_{\mathbb{R}^3} \left(2q^0\Phi+\Phi^2\right)J^0\frac{dq}{q^0}+4\int_{\mathbb{R}^3} \left( u_{\mu} q^\mu\right)^2 f\sqrt{J^0}\frac{dq}{q^0}-\int_{\mathbb{R}^3} f\sqrt{J^0}\frac{dq}{q^0}}{3\widetilde{h}(\beta_0)(n h ) } \right\}\mathbf{q}\cdot q\sqrt{J^0},\cr
 	\Gamma_{4} (f)&=-\beta_0q^0(U^0-1)\sqrt{J^0},\cr
 		\Gamma_{5} (f)&=\frac{1}{\sqrt{J^0}}\int_0^1 (1-\theta)(n-1,U^0-1,U,e-e_0)D^2J(\theta)(n-1,U^0-1,U,e-e_0)^Td\theta,
 	\end{align*}
where $\Gamma_{3}^*$ is given by
 \begin{align*}
 	\Gamma_{3}^*(f)	&=-\sum_{i=1}^3\int_{\mathbb{R}^3} q^if\sqrt{J^0}\frac{dq}{q^0}\int_{\mathbb{R}^3} q^i qf\sqrt{J^0}\frac{dq}{q^0}+\int_{\mathbb{R}^3} \Phi_1 qF\frac{dq}{q^0}\cr
 	&+e_0\int_{\mathbb{R}^3} q f\sqrt{J^0}\frac{dq}{q^0}\left\{\frac{\Psi}{2}+\frac{\Psi^3-3\Psi^2}{2(2+\Psi-\Psi^2+2\sqrt{1+\Psi})}\right\}\cr
 	& -\frac{1}{n}\left(\int_{\mathbb{R}^3} q^0 f\sqrt{J^0} dq+\int_{\mathbb{R}^3}\left(2q^0\Phi+\Phi^2 \right)F\frac{dq}{q^0} \right)\int_{\mathbb{R}^3} q f\sqrt{J^0}\frac{dq}{q^0},
 \end{align*}
 and $\Psi,\Psi_1,\Phi$ and $\Phi_1$ denote
 \begin{align}\label{notation2}
 	\begin{split}
 		\Psi&=2\int_{\mathbb{R}^3} f\sqrt{J^0}\,dq+\left(\int_{\mathbb{R}^3} f\sqrt{J^0}\,dq\right)^2-\sum_{i=1}^3\left(\int_{\mathbb{R}^3}f\sqrt{J^0}q^i\,\frac{dq}{q^0}\right)^2,\cr
 		\Psi_1&= \left(\int_{\mathbb{R}^3} f\sqrt{J^0}\,dq\right)^2-\sum_{i=1}^3\left(\int_{\mathbb{R}^3}f\sqrt{J^0}q^i\,\frac{dq}{q^0}\right)^2,\cr
 		\Phi&= u_{\mu} q^\mu-q^0,\cr
 		\Phi_1&= u_{\mu} q^\mu-q^0+\sum_{i=1}^3 q^i\int_{\mathbb{R}^3}q^i f\sqrt{J^0}\,\frac{dq}{q^0}.
 \end{split}\end{align}
 \end{lemma}
 \begin{proof}
Define the transitional macroscopic fields between $F$ and $J^0$:
$$
\left(n_{\theta},U^0_{\theta},U_{\theta},  e_{\theta} \right) =\theta\left( n ,U^0,U , e\right) +(1-\theta)\left(1,1, 0, e_0\right),
$$
and the transitional relativistic Maxwellian:
$$
J(\theta)=\frac{n_\theta}{M(\beta_\theta)} e^{-\beta_\theta U_\theta^\mu q_\mu},\quad \text{where}\ \beta_\theta=(\widetilde{e})^{-1}(e_\theta).
$$
Then $J(F)$ and $J^0$ can be rewritten by
\begin{equation*}
J(F)=J(n,U^0,U,e )\equiv J(1), \qquad J^0=J(1,1,0,e_0 )\equiv J(0).
\end{equation*}	
We then apply Taylor expansion to have
 \begin{align}\label{taylor}
 \begin{split}
 J (F)-J^0&=J(1)-J(0)\cr
 &=J^\prime(0)+\int_{0}^1(1-\theta)J^{\prime\prime}(\theta) \,d\theta\cr
 &= \frac{\partial J(0)}{\partial n_\theta}\frac{\partial n_\theta}{\partial \theta}\bigg|_{\theta=0}+\frac{\partial J(0)}{\partial U^0_\theta}\frac{\partial U^0_\theta}{\partial \theta}\bigg|_{\theta=0}+\nabla_{U_\theta}J(0)\cdot \frac{\partial U_\theta}{\partial \theta}\bigg|_{\theta=0}+\frac{\partial J(0)}{\partial e_\theta}\frac{\partial e_\theta}{\partial \theta}\bigg|_{\theta=0}\cr
 &+ \int_0^1 (1-\theta)(n -1,U ^0-1,U ,e -e_0)D^2J(\theta)(n -1,U ^0-1,U ,e -e_0)^Td\theta.
\end{split}
 \end{align}
Now, we compute
\begin{align*}
		\frac{\partial J(\theta)}{\partial n_\theta} =\frac{1}{n_\theta}&J(\theta),\qquad \frac{\partial J(\theta)}{\partial U^0_\theta}=-\beta_\theta q^0J(\theta), \qquad\nabla_{U_\theta}J(\theta)=\beta_\theta qJ(\theta),\cr
		&\frac{\partial J(\theta)}{\partial e_\theta} =-\frac{1}{\left\{\widetilde{e}\right\}^{\prime}(\beta_\theta)}\left(\frac{M^\prime(\beta_\theta)}{M\left(\beta_\theta\right)}+U_\theta^\mu q_\mu \right)J(\theta),
\end{align*}
so that
\begin{align}\label{derivative2}
	\begin{split}
	\frac{\partial J(0)}{\partial n_\theta}=&J^0,\qquad \frac{\partial J(0)}{\partial U^0_\theta}=-\beta_0 q^0J^0,\qquad \nabla_{U_\theta}J(0)=\beta_0 qJ^0,\cr
	&\frac{\partial J(0)}{\partial e_\theta}=-\frac{1}{\left\{\widetilde{e}\right\}^{\prime}(\beta_0)}\left(q^0-e_0\right)J^0.
\end{split}\end{align}
In the last line, we used Lemma \ref{KM}:
$$
\frac{M^\prime(\beta_0)}{M(\beta_0)}=-\frac{K_1(\beta_0)}{K_2(\beta_0)}-\frac{3}{\beta_0} =-e_0.
$$
Inserting \eqref{derivative2} into \eqref{taylor}, we derive
\begin{align*}	
 		 	\frac{J (F)-J^0}{\sqrt{J^0}}&=(n -1)\sqrt{J^0}-\frac{1}{\left\{\widetilde{e}\right\}^{\prime}(\beta_0)}(e -e_0)(q^0-e_0)\sqrt{J^0}+\beta_0U \cdot q\sqrt{J^0}-\beta_0q^0(U ^0-1)\sqrt{J^0}\cr
 		&+\frac{1}{\sqrt{J^0}}\int_0^1 (1-\theta)(n -1,U ^0-1,U ,e -e_0)D^2J(\theta)(n -1,U ^0-1,U ,e -e_0)^Td\theta\cr
 		&=I_1+I_2+I_3+I_4+I_5.
 	 	\end{align*}	
We consider each $I_i$ $(i=1,\cdots,5)$ separately. Note in the following that $I_1,I_2,I_3$ are decomposed into the linear part and the nonlinear part.
 	\newline
 	
\noindent 	$\bullet$ Decomposition of $I_1$: Inserting $F=J^0+f\sqrt{J^0}$, a direct computation gives
\begin{align}\label{n=}
	\begin{split}
		n&=\left\{\left(\int_{\mathbb{R}^3} Fdq\right)^2-\sum_{i=1}^3\left(\int_{\mathbb{R}^3}Fq^i\frac{dq}{q^0}\right)^2\right\}^{\frac{1}{2}}\cr
		&=\left\{ 1+2\int_{\mathbb{R}^3} f\sqrt{J^0}\,dq+\left(\int_{\mathbb{R}^3} f\sqrt{J^0}\,dq\right)^2-\sum_{i=1}^3\left(\int_{\mathbb{R}^3}f\sqrt{J^0}q^i\,\frac{dq}{q^0}\right)^2\right\}^{\frac{1}{2}}\cr
		&= \sqrt{1+\Psi}.
\end{split}	\end{align}
To extract the linear part from $n$, we recall the following identity \cite{BCNS}:
 \begin{equation}\label{route pi}
 \sqrt{1+\Psi}=1+ \frac{\Psi}{2}-\frac{\Psi^2}{2(2+\Psi+2\sqrt{1+\Psi})}.
\end{equation}
Using this identity together with \eqref{notation2} and \eqref{n=} gives
  \begin{align*}
 			(n-1)\sqrt{J^0}&=\left(\sqrt{1+\Psi}-1\right)\sqrt{J^0}\cr
&=\left( \frac{\Psi}{2}-\frac{\Psi^2}{2(2+\Psi+2\sqrt{1+\Psi})}
\right)\sqrt{J^0}\cr
 			&= \left(\int_{\mathbb{R}^3} f\sqrt{J^0}\,dq+\frac{\Psi_1}{2}-\frac{\Psi^2}{2(2+\Psi+2\sqrt{1+\Psi})}\right)\sqrt{J^0}\cr
&=  \int_{\mathbb{R}^3} f\sqrt{J^0}\,dq\sqrt{J^0}+\Gamma_1(f).
	 	\end{align*}
 	\newline
 	
\noindent 	$\bullet$ Decomposition of $I_2$: First, we consider the following identity \cite{BCNS}:
\begin{align*}
	\frac{1}{\sqrt{1+\Psi}}=1-\frac{\Psi}{2}-\frac{\Psi^3-3\Psi^2}{2(2+\Psi-\Psi^2+2\sqrt{1+\Psi})},
\end{align*}
to decompose $1/n$  as follows:
\begin{align}\label{1/n=}
\begin{split}	\frac{1}{n}&=1-\frac{\Psi}{2}-\frac{\Psi^3-3\Psi^2}{2(2+\Psi-\Psi^2+2\sqrt{1+\Psi})}\cr
	&=	1-\int_{\mathbb{R}^3} f\sqrt{J^0}\,dq-\frac{\Psi_1}{2}	-\frac{\Psi^3-3\Psi^2}{2(2+\Psi-\Psi^2+2\sqrt{1+\Psi})}.
\end{split}
\end{align}
This identity, together with $\eqref{notation2}_3$ enables one to express $e-e_0$ as
\begin{align}\label{e-e3}
 \begin{split}
	e-e_0	&=\frac{1}{n}\int_{\mathbb{R}^3} (u^\mu  q_\mu)^2 F\frac{dq}{q^0}-e_0\cr
	&=	\left\{	1-\int_{\mathbb{R}^3} f\sqrt{J^0}\,dq-\frac{\Psi_1}{2}	-\frac{\Psi^3-3\Psi^2}{2(2+\Psi-\Psi^2+2\sqrt{1+\Psi})}\right\}\cr
	&\quad\times\int_{\mathbb{R}^3} \left\{	(q^0)^2+2q^0\Phi+\Phi^2\right\}F\frac{dq}{q^0}-e_0\cr
	&\equiv\left\{\left( 1-\int_{\mathbb{R}^3} f\sqrt{J^0}\,dq \right)\int_{\mathbb{R}^3}(q^0)^2 F\,\frac{dq}{q^0}-e_0  \right\}+R_{I_2}(f).
\end{split}
\end{align}
We then extract the linear part from the above expression:
 \begin{align}\label{e-e4}
 	\begin{split}
 	&\left(	1-\int_{\mathbb{R}^3} f\sqrt{J^0}\,dq\right)\int_{\mathbb{R}^3} (q^0)^2 F\frac{dq}{q^0}-e_0\cr
 	&\qquad= 	\left(	1-\int_{\mathbb{R}^3} f\sqrt{J^0}\,dq\right)\int_{\mathbb{R}^3} q^0\left( J^0 +f\sqrt{J^0}\right)\,dq-e_0\cr
 	&\qquad=e_0+\int_{\mathbb{R}^3} q^0 f\sqrt{J^0} \,dq-e_0\int_{\mathbb{R}^3}  f\sqrt{J^0} \,dq-\int_{\mathbb{R}^3}  f\sqrt{J^0} \,dq\int_{\mathbb{R}^3} q^0 f\sqrt{J^0} \,dq-e_0\cr
 	&\qquad=\int_{\mathbb{R}^3} \left(q^0-e_0\right) f\sqrt{J^0} \,dq-\int_{\mathbb{R}^3}  f\sqrt{J^0} \,dq\int_{\mathbb{R}^3} q^0 f\sqrt{J^0} \,dq,
\end{split}\end{align}
to write \eqref{e-e3} as
\begin{align}\label{e-e0}		
 				e-e_0&= \int_{\mathbb{R}^3} \left(q^0-e_0\right)f\sqrt{J^0}\,dq+\left\{-\int_{\mathbb{R}^3} f\sqrt{J^0}\,dq \int_{\mathbb{R}^3} q^0f\sqrt{J^0}\,dq+R_{I_2}(f)\right\}.
 	\end{align}
 Therefore, we obtain the following decomposition of $I_2$:
 	\begin{align*}
 		I_2&=-\frac{1}{\left\{\widetilde{e}\right\}^{\prime}(\beta_0)}(e-e_0)\left(q^0-e_0\right)\sqrt{J^0}\cr
 		&=-\frac{1}{\left\{\widetilde{e}\right\}^{\prime}(\beta_0)}\int_{\mathbb{R}^3} \left(q^0-e_0\right) f\sqrt{J^0}\,dq\left(q^0-e_0 \right)\sqrt{J^0}+\Gamma_2 (f).
 	\end{align*}
 	\newline

\noindent 	$\bullet$ Decomposition of $I_3$: We recall from \eqref{LLD1} that $nh$ takes the form of
 			$$
 			nh=\frac{4}{3}\int_{\mathbb{R}^3} \left(u^\mu q_\mu\right)^2 F\frac{dq}{q^0}-\frac{1}{3}\int_{\mathbb{R}^3} F\frac{dq}{q^0}
 			$$
to derive from Lemma \ref{J^0} and $\eqref{notation2}_3$ that
 		\begin{align}\label{nh0}
 			\begin{split}
 			nh	&=\frac{4}{3}\int_{\mathbb{R}^3} \left\{(q^0)^2+2q^0\Phi+\Phi^2\right\} J^0\frac{dq}{q^0}+\frac{4}{3}\int_{\mathbb{R}^3} ( u_{\mu} q^\mu)^2 f\sqrt{J^0}\frac{dq}{q^0}-\frac{1}{3}\int_{\mathbb{R}^3}\left(  J^0+f\sqrt{J^0}\right)\frac{dq}{q^0}\cr
 			&=\frac{4}{3}e_0+\frac{4}{3}\int_{\mathbb{R}^3} \left(2q^0\Phi+\Phi^2\right)J^0\frac{dq}{q^0}+\frac{4}{3}\int_{\mathbb{R}^3} ( u_{\mu} q^\mu)^2 f\sqrt{J^0}\frac{dq}{q^0}-\frac{1}{3}\frac{K_1(\beta_0)}{K_2(\beta_0)}-\frac{1}{3}\int_{\mathbb{R}^3} f\sqrt{J^0}\frac{dq}{q^0} \cr
 			&=\widetilde{h}(\beta_0)+\frac{4}{3}\int_{\mathbb{R}^3} \left(2q^0\Phi+\Phi^2\right)J^0\frac{dq}{q^0}+\frac{4}{3}\int_{\mathbb{R}^3} ( u_{\mu} q^\mu)^2 f\sqrt{J^0}\frac{dq}{q^0}-\frac{1}{3}\int_{\mathbb{R}^3} f\sqrt{J^0}\frac{dq}{q^0}.
 		 \end{split}	\end{align}
 	 In the last line, we used
 	 $$
 	 \frac{4}{3}e_0-\frac{1}{3}\frac{K_1(\beta_0)}{K_2(\beta_0)}= 	 \frac{4}{3}\left(\frac{K_1(\beta_0)}{K_2(\beta_0)}+\frac{3}{\beta_0}\right)-\frac{1}{3}\frac{K_1(\beta_0)}{K_2(\beta_0)}= \widetilde{h}(\beta_0).
 	 $$
 From this, we can express $1/(nh)$ by
 	 	\begin{align*}
 			\frac{1}{nh}&=\frac{1}{\widetilde{h}(\beta_0)}-\frac{4\int_{\mathbb{R}^3} \left(2q^0\Phi+\Phi^2\right)J^0\frac{dq}{q^0}+4\int_{\mathbb{R}^3} ( u_{\mu} q^\mu)^2 f\sqrt{J^0}\frac{dq}{q^0}-\int_{\mathbb{R}^3} f\sqrt{J^0}\frac{dq}{q^0}}{3\widetilde{h}(\beta_0)(n h ) }\cr
 		&\equiv \frac{1}{\widetilde{h}(\beta_0)}+R_{I_3}(f).
 	\end{align*}
This leads to
 \begin{align}\label{U_F}
 \begin{split}
 			U&=u+\frac{\mathbf{q}}{nh}\cr
 			&=\int_{\mathbb{R}^3} q f\sqrt{J^0}\frac{dq}{q^0}-\left\{\frac{H}{2}+\frac{\Psi^3-3\Psi^2}{2(2+\Psi-\Psi^2+2\sqrt{1+\Psi})}\right\}\int_{\mathbb{R}^3} q f\sqrt{J^0}\frac{dq}{q^0} + \frac{\mathbf{q}}{\widetilde{h}(\beta_0)}+R_{I_3}(f)\mathbf{q} .
 		\end{split}
 	\end{align}
 Note that we used \eqref{1/n=}:
 	\begin{align}\label{u_F}
 		\begin{split}
 			u &=\frac{1}{n }\int_{\mathbb{R}^3} qF\frac{dq}{q^0}\cr
 &=\left\{1-\frac{\Psi}{2}-\frac{\Psi^3-3\Psi^2}{2(2+\Psi-\Psi^2+2\sqrt{1+\Psi})}\right\}\int_{\mathbb{R}^3} q f\sqrt{J^0}\frac{dq}{q^0}\cr
 &=\int_{\mathbb{R}^3} q f\sqrt{J^0}\frac{dq}{q^0}-\left\{\frac{\Psi}{2}+\frac{\Psi^3-3\Psi^2}{2(2+\Psi-\Psi^2+2\sqrt{1+\Psi})}\right\}\int_{\mathbb{R}^3} q f\sqrt{J^0}\frac{dq}{q^0}.
 	\end{split}\end{align}
 	We now focus on $\mathbf{q}$ to extract the linear part, which is defined in $\eqref{LLD1}_2$ by
 	\begin{equation}\label{heat}
 	\mathbf{q}=\int_{\mathbb{R}^3}( u_{\mu} q^\mu) q F\,\frac{dq}{q^0}-u \int_{\mathbb{R}^3} ( u_{\mu} q^\mu)^2 F\,\frac{dq}{q^0}.
 	\end{equation}
 	Recall from $\eqref{notation2}_4$ that
$$
 u_{\mu} q^\mu=q^0-\sum_{i=1}^3 q^i\int q^i f\sqrt{J^0}\frac{dq}{q^0}+\Phi_1,
$$
and insert this into the first term of \eqref{heat} to derive
 	\begin{align}\label{flux1}
 		\begin{split}
 			\int_{\mathbb{R}^3}( u_{\mu} q^\mu) q F\frac{dq}{q^0}	&= \int_{\mathbb{R}^3} \biggl(q^0-\sum_{i=1}^3 q^i\int q^i f\sqrt{J^0}\frac{dq}{q^0}+\Phi_1\biggl) q \left(J^0+f\sqrt{J^0}\right)\frac{dq}{q^0}\cr
 			&=\int_{\mathbb{R}^3}  q f\sqrt{J^0}dq-\frac{1}{\beta_0}\int_{\mathbb{R}^3} qf\sqrt{J^0}\frac{dq}{q^0}\cr
 			& -\sum_{i=1}^3\int_{\mathbb{R}^3} q^if\sqrt{J^0}\frac{dq}{q^0}\int_{\mathbb{R}^3} q^i qf\sqrt{J^0}\frac{dq}{q^0}+\int_{\mathbb{R}^3} \Phi_1 qF\frac{dq}{q^0}.
 		\end{split}
 	\end{align}
 Here we used Lemma \ref{J^0} so that
 \begin{align*}
 	\sum_{i=1}^3 \int_{\mathbb{R}^3} q^i f\sqrt{J^0}\,\frac{dq}{q^0}\int_{\mathbb{R}^3} q^i qJ^0\,\frac{dq}{q^0}=\frac{1}{\beta_0}\int_{\mathbb{R}^3} qf\sqrt{J^0}\,\frac{dq}{q^0}.
 \end{align*}
On the other hand, in view of $\eqref{notation2}_3$, we compute the second term of \eqref{heat} as
 \begin{align}\label{flux2}
 	\begin{split}
 		\int_{\mathbb{R}^3} ( u_{\mu} q^\mu)^2F\frac{dq}{q^0}	&=\int_{\mathbb{R}^3} \left\{(q^0)^2+2q^0\Phi+\Phi^2 \right\}(J^0+f\sqrt{J^0})\,\frac{dq}{q^0}\cr
 		&=e_0+\int_{\mathbb{R}^3} q^0 f\sqrt{J^0} \,dq+\int_{\mathbb{R}^3}\left(2q^0\Phi+\Phi^2 \right)F\,\frac{dq}{q^0}.
 	\end{split}
 \end{align}
We now go back to \eqref{heat} with \eqref{flux1} and \eqref{flux2} to get
  \begin{align}\label{heat2}
  	\begin{split}
 \mathbf{q}	&= \int_{\mathbb{R}^3}  q f\sqrt{J^0}dq-\frac{1}{\beta_0}\int_{\mathbb{R}^3} qf\sqrt{J^0}\frac{dq}{q^0}-\sum_{i=1}^3\int_{\mathbb{R}^3} q^if\sqrt{J^0}\frac{dq}{q^0}\int_{\mathbb{R}^3} q^i qf\sqrt{J^0}\frac{dq}{q^0}+\int_{\mathbb{R}^3} \Phi_1 qF\frac{dq}{q^0}\cr
 	&-\left\{1-\frac{\Psi}{2}-\frac{\Psi^3-3\Psi^2}{2(2+\Psi-\Psi^2+2\sqrt{1+\Psi})}\right\}\int_{\mathbb{R}^3} q f\sqrt{J^0}\frac{dq}{q^0}\cr
 	&\quad\times \left(e_0+\int_{\mathbb{R}^3} q^0 f\sqrt{J^0} \,dq+\int_{\mathbb{R}^3}\left(2q^0\Phi+\Phi^2 \right)F\,\frac{dq}{q^0}\right)\cr
 	&=	\int_{\mathbb{R}^3}  q f\sqrt{J^0}dq-\left(\frac{1}{\beta_0}+e_0\right)\int_{\mathbb{R}^3} qf\sqrt{J^0}\frac{dq}{q^0}+\Gamma_{3}^*(f)\cr
 	&=	\int_{\mathbb{R}^3}  q f\sqrt{J^0}dq-\widetilde{h}(\beta_0)\int_{\mathbb{R}^3} qf\sqrt{J^0}\frac{dq}{q^0}+\Gamma_{3}^*(f).
\end{split} \end{align}
Plugging \eqref{heat2} into  \eqref{U_F} gives
	\begin{align*}
		U&=\int_{\mathbb{R}^3} q f\sqrt{J^0}\frac{dq}{q^0}-\left\{\frac{\Psi}{2}+\frac{\Psi^3-3\Psi^2}{2(2+\Psi-\Psi^2+2\sqrt{1+\Psi})}\right\}\int_{\mathbb{R}^3} q f\sqrt{J^0}\frac{dq}{q^0}\cr
		& + \frac{1}{\widetilde{h}(\beta_0)}\left(	\int_{\mathbb{R}^3}  q f\sqrt{J^0}dq-\widetilde{h}(\beta_0)\int_{\mathbb{R}^3} qf\sqrt{J^0}\frac{dq}{q^0}+\Gamma_{3}^*(f)\right)+R_{I_3}(f)\mathbf{q} \cr
		&=\frac{1}{\widetilde{h}(\beta_0)}\int_{\mathbb{R}^3}  q f\sqrt{J^0}dq -\left\{\frac{\Psi}{2}+\frac{\Psi^3-3\Psi^2}{2(2+\Psi-\Psi^2+2\sqrt{1+\Psi})}\right\}\int_{\mathbb{R}^3} q f\sqrt{J^0}\frac{dq}{q^0}\cr
		& +\frac{1}{\widetilde{h}(\beta_0)}\Gamma_{3}^*(f) +R_{I_3}(f)\mathbf{q},
 \end{align*}	
and thus we have
 	\begin{align*}
 		I_3=\beta_0U \cdot q\sqrt{J^0}&= \frac{\beta_0}{\widetilde{h}(\beta_0)}\int qf\sqrt{J^0}\, dq\cdot q\sqrt{J^0}+\Gamma_{3}(f).
 	\end{align*}

\noindent $\bullet$ $I_4,I_5$: We note that $I_4=\Gamma_{4}(f)$ and $I_5=\Gamma_{5}(f)$. This completes the proof.
 \end{proof}
\noindent
\newline

The following proposition gives the linearized Anderson-Witting model:
\begin{proposition}\label{lin3}
	For the solution $F=J^0+f\sqrt{J^0}$ to the Anderson-Witting model \eqref{AWRBGK1}, the perturbation $f$ verifies
	\begin{align}\label{LAW}\begin{split}
			\partial_t f+\hat{q}\cdot\nabla_x f&=L(f)+\Gamma (f),\cr
f_0(x,q)&=f(0,x,q),
\end{split}\end{align}
	where the linearized operator $L(f)$ is defined by $P(f)-f$, and the nonlinear perturbation $\Gamma (f)$ is defined by
	\begin{align*}
		\Gamma (f)&=\frac{U_\mu q^\mu}{q^0} \sum_{i=1}^5\Gamma_i (f)+\frac{1}{q^0}\left(\Phi+ \frac{\mathbf{q}_\mu q^\mu}{nh} \right) L(f).
	\end{align*}
\end{proposition}
\begin{proof}
Insert $F=J^0+f\sqrt{J^0}$ into \eqref{AWRBGK1} to obtain
	\begin{align*}
		\partial_t f+\hat{q}\cdot\nabla_x f&=\frac{U_\mu q^\mu}{q^0}\left(\frac{J(F)-J^0}{\sqrt{J^0}}-f\right).
	\end{align*}
	Using $\eqref{notation2}_3$, we  express $ U_\mu q^\mu/q^0$ as
	\begin{align*}
		\frac{U_\mu q^\mu}{q^0}&=\frac{1}{q^0}\left(u_\mu q^\mu+\frac{\mathbf{q} _\mu q^\mu}{n h }\right)\cr
		&=\frac{1}{q^0}\left(q^0+\Phi +\frac{\mathbf{q}_\mu q^\mu}{n h }\right)\cr
		&=1+\frac{1}{q^0}\left( \Phi+ \frac{\mathbf{q}_\mu q^\mu}{nh} \right),
	\end{align*}
which, together with Lemma \ref{lin2}, gives the desired result.	
\end{proof}
 \noindent
\subsection{Analysis of the linearized operator}
	Let $N$ be the five dimensional space spanned by $\{\sqrt{J^0},q^\alpha \sqrt{J^0}\}$. Denote $e_i$ $(i=1,\cdots,5)$ by
		$$
	e_1=\sqrt{J^0},\qquad  e_{2,3,4}=\sqrt{\frac{\beta_0}{\widetilde{h}(\beta_0)}}q\sqrt{J^0},  \qquad e_5=\sqrt{-\frac{1}{\left\{\widetilde{e}\right\}^{\prime}(\beta_0)}}\left(q^0-e_0\right)\sqrt{J^0},
	$$
so that $P(f)$ can be written by
	\begin{align}\label{Pf}
		P(f)=\langle f,e_1\rangle_q e_1+\langle f,e_{2,3,4} \rangle_q \cdot e_{2,3,4}+ \langle f,e_{5}\rangle_q e_5.
	\end{align}
Notice that $e_5$ is well defined since, by Lemma \ref{beta decreasing}, we have
	\begin{align*}
		-\left\{\widetilde{e}\right\}^{\prime}(\beta_0)&= -\frac{d}{d\beta}\left\{\frac{K_1(\beta)}{K_2(\beta)}+\frac{3}{\beta} \right\}_{\beta=\beta_0}>0.
	\end{align*}
\begin{lemma}\label{ortho}
$P$ is an orthonormal projection from $L_q^2(\mathbb{R}^3)$ onto $N$.
\end{lemma}
\begin{proof}
	It is enough to show that $\{e_i\}$ $(i=1,\cdots,5)$ forms an orthonormal basis.
	\newline
$\bullet$ $\|e_1\|_{L^2_q}=1$: It is straightforward from the definition of $J^0$ that
\begin{equation*}
	\langle e_1,e_1\rangle_q =\int_{\mathbb{R}^3}  J^0dq=1.
\end{equation*}
$\bullet$ $\|e_{i+1}\|_{L^2_q}=1$ $(i=1,2,3)$: A direct computation, using the spherical coordinates and integration by parts,  gives
	\begin{align*}
		\int_{\mathbb{R}^3} |q|^2 J^0 dq&=\frac{1}{M(\beta_0)} \int_{\mathbb{R}^3} |q|^2e^{-\beta_0\sqrt{1+|q|^2}}\,dq\cr
	&=\frac{4\pi}{M(\beta_0)} \int_0^\infty r^4e^{-\beta_0\sqrt{1+r^2}}dr\cr
	&=\frac{4\pi}{\beta_0^2M(\beta_0)} \left(3\int_0^\infty e^{-\beta_0\sqrt{1+|r|^2}}  \,dr+\frac{12}{\beta_0}\int_0^\infty e^{-\beta_0\sqrt{1+|r|^2}} \frac{2r^2+1}{\sqrt{1+r^2}} dr \right).
\end{align*}
This, together with \eqref{CV} and \eqref{K_2}, leads to
\begin{equation}\label{|q|^2J0}
	\int_{\mathbb{R}^3} |q|^2 J^0 dq=\frac{12}{\beta^2_0}+\frac{3}{\beta_0}\frac{K_1(\beta_0)}{K_2(\beta_0)}.
\end{equation}
Therefore,
	\begin{align*}
 \langle e_{i+1},e_{i+1}\rangle_q&=\frac{\beta_0}{\widetilde{h}(\beta_0)}\int_{\mathbb{R}^3} (q^i)^2 J^0dq\cr
 &=\left(\frac{1}{\beta_0}\frac{K_1(\beta_0)}{K_2(\beta_0)}+\frac{4}{\beta_0^2}\right)^{-1}\int_{\mathbb{R}^3} \frac{1}{3}|q|^2J^0dq\cr
 &=1.
	\end{align*}
$\bullet$ $\|e_5\|_{L^2_q}=1$: We use Lemma \ref{Bessel2} to obtain
\begin{align*}
	-\left\{\widetilde{e}\right\}^{\prime}(\beta_0)&=-\frac{d}{d\beta}\left\{\frac{K_1(\beta)}{K_2(\beta)}+\frac{3}{\beta}  \right\}_{\beta=\beta_0}\cr
	&=-\frac{3}{\beta_0}\frac{K_1(\beta_0)}{K_2(\beta_0)}-\left(\frac{K_1(\beta_0)}{K_2(\beta_0)}\right)^2 +1+\frac{3}{\beta_0^2}.
\end{align*}
Then, we have from Lemma \ref{J^0} and \eqref{|q|^2J0} that
	\begin{align*}
\langle e_5,e_5\rangle_q&=-\frac{1}{\left\{\widetilde{e}\right\}^{\prime}(\beta_0)}\int_{\mathbb{R}^3}(q^0-e_0)^2J^0dq\cr
		&=-\frac{1}{\left\{\widetilde{e}\right\}^{\prime}(\beta_0)}\left(\int_{\mathbb{R}^3}\left(1+|q|^2\right)J^0dq-2e_0\int_{\mathbb{R}^3}q^0J^0dq+e_0^2\int_{\mathbb{R}^3}J^0dq\right)\cr
		&=-\frac{1}{\left\{\widetilde{e}\right\}^{\prime}(\beta_0)}\left\{1+\frac{12}{\beta^2_0}+\frac{3}{\beta_0}\frac{K_1(\beta_0)}{K_2(\beta_0)}- \left(\frac{K_1(\beta_0)}{K_2(\beta_0)}+\frac{3}{\beta_0}\right)^2 \right\}\cr
		&=1.	\end{align*}
$\bullet$ $\langle e_i,e_j\rangle_q=0$ $(i\neq j)$: The orthogonality can be proved in a similar manner. We omit the proof.
	
\end{proof}
We are ready to prove the dissipative property of $L$.
\begin{proposition}\label{pro}
	The linearized operator $L:=P-I $ satisfies the following properties:
	\begin{enumerate}
		\item $Ker(L)=N$.
		\item $L$ is dissipative in the following sense:
		$$
		\langle L(f),f\rangle_q= -\|\{I-P\}f\|^2_{L^2_q} \le 0.
		$$
	\end{enumerate}
\end{proposition}
\begin{proof}
	(1) follows from the definition of $P$. To prove (2), we use the orthonormality of $P$ to see that
	\begin{align*}
	\langle P(f),\{I-P\}(f)\rangle_q&=	\langle P(f),f\rangle_q-	\langle P(f),P(f)\rangle_q\cr
	&=\sum_{i=1}^5 \left|\langle f, e_i\rangle_q\right|^2-\left\langle \sum_{i=1}^5 \langle f, e_i\rangle_qe_i,~\sum_{j=1}^5 \langle f, e_j\rangle_q e_j  \right\rangle_q\cr
	&=0.
	\end{align*}
Thus, we have
	\begin{align*}
	\langle L(f),f\rangle_q&=	-\langle \{I-P\}(f),~P(f)+\{I-P\}(f)\rangle_q\cr
	&=-\langle \{I-P\}(f),~ \{I-P\}(f)\rangle_q\cr
	&=-\|\{I-P\}(f)\|^2_{L^2_q}.
	\end{align*}


\end{proof}
\bigskip
\section{Estimates on the macroscopic fields and the nonlinear terms}

In this section, we study the estimates on the nonlinear perturbation $\Gamma$ necessary for local in time existence and energy estimates. We start with the estimates of the macroscopic fields.
\subsection{Estimates of macroscopic fields} We first need to estimate $\Psi, \Psi_1,\Phi$ and $\Phi_1$ whose definition is given in \eqref{notation2}.
\begin{lemma}\label{lem22}
Suppose $E(f)(t)$ is sufficiently small. Then $\Psi, \Psi_1,\Phi$ and $\Phi_1$ satisfy
\begin{enumerate}
	\item $\displaystyle |\partial^\alpha \Psi|+|\partial^\alpha \Phi|\le C(1+q^0)\sum_{|k|\le|\alpha|}\|\partial^k f\|_{L^2_q}.$
	\item $\displaystyle |\partial^\alpha \Psi_1|+|\partial^\alpha \Phi_1|\le C(1+q^0)\sqrt{E(t)}\sum_{|k|\le|\alpha|}\|\partial^k f\|_{L^2_q} .$
\end{enumerate}
\end{lemma}
\begin{proof}
$\bullet$ Estimates of $\Psi$ and $\Psi_1$: From H\"{o}lder inequality, we have
	\begin{align}\label{Psi}
	\begin{split}
|\partial^\alpha \Psi|&=\left|\partial^\alpha \left\{2\int_{\mathbb{R}^3} f\sqrt{J^0}\,dq+\left(\int_{\mathbb{R}^3} f\sqrt{J^0}\,dq\right)^2-\sum_{i=1}^3\left(\int_{\mathbb{R}^3}f\sqrt{J^0}q^i\,\frac{dq}{q^0}\right)^2 \right\}\right|\cr
			&\le 2\|\partial^\alpha f\|_{L^2_q}+ C\sum_{|k|\le |\alpha|}\|\partial^k f\|_{L^2_q}\|\partial^{\alpha-k}f\|_{L^2_q}.
\end{split}	\end{align}
Then we apply the Sobolev embedding $H^2(\mathbb{T}^{3})\subseteq L^{\infty}(\mathbb{T}^{3})$ to lower order terms to get
\begin{align}\label{H}
	\begin{split}
|\partial^\alpha \Psi| &\le 2\|\partial^\alpha f\|_{L^2_q}+C\sqrt{E(t)}\sum_{|k|\le |\alpha|}\|\partial^k f\|_{L^2_q}\le C\sum_{|k|\le |\alpha|}\|\partial^k f\|_{L^2_q}.
\end{split}\end{align}
Similarly, we have
	\begin{align}\label{Psi1}
	\begin{split}
		|\partial^\alpha \Psi_1|&=\left|\partial^\alpha \left\{\left(\int_{\mathbb{R}^3} f\sqrt{J^0}\,dq\right)^2-\sum_{i=1}^3\left(\int_{\mathbb{R}^3}f\sqrt{J^0}q^i\,\frac{dq}{q^0}\right)^2 \right\}\right|\cr
		&\le C\sum_{|k|\le|\alpha|}\|\partial^k f\|_{L^2_q}\|\partial^{\alpha-k} f\|_{L^2_q}\cr
		&\le C\sqrt{E(t)}\sum_{|k|\le |\alpha|}\|\partial^k f\|_{L^2_q}.
\end{split} \end{align}
$\bullet$ Estimates of $\Phi$: We observe that
\begin{align*}
	\begin{split}
	\Phi=u^\mu q_\mu-q^0=\frac{1}{n}\int_{\mathbb{R}^3}q^\mu F\,\frac{dq}{q^0} q_\mu -q^0.
\end{split}\end{align*} 	
Inserting \eqref{1/n=} and $F=J^0+f\sqrt{J^0}$, we have
	\begin{align*}
	\Phi&=-\sum_{i=1}^3q^i \int_{\mathbb{R}^3} f\sqrt{J^0} q^i\,\frac{dq}{q^0}-q^0\left(\int_{\mathbb{R}^3}f\sqrt{J^0} \,dq\right)^2+\sum_{i=1}^3q^i \int_{\mathbb{R}^3} f\sqrt{J^0} q^i\,\frac{dq}{q^0}\int_{\mathbb{R}^3}f\sqrt{J^0} \,dq\cr
	&-\frac{1}{2}\left\{\left(\int_{\mathbb{R}^3} f\sqrt{J^0}dq\right)^2-\sum_{i=1}^3\left(\int_{\mathbb{R}^3}f\sqrt{J^0}q^i\frac{dq}{q^0}\right)^2
	+\frac{\Psi^3-3\Psi^2}{(2+\Psi-\Psi^2+2\sqrt{1+\Psi})} \right\}\cr
	&\quad\times\left(q^0+ 	q^0\int_{\mathbb{R}^3}f\sqrt{J^0} \,dq-\sum_{i=1}^3q^i \int_{\mathbb{R}^3} f\sqrt{J^0} q^i\,\frac{dq}{q^0} \right).
\end{align*}
Then it follows from \eqref{H} that
\begin{align*}
	|\Phi|&\le C(1+q^0)\left(\|f\|_{L^2_q}+\|f\|_{L^2_q}^2\right)+C\left(\|f\|_{L^2_q}^2
	+\frac{\|f\|_{L^2_q}^3+\|f\|_{L^2_q}^2}{2-\|f\|_{L^2_q}-\|f\|_{L^2_q}^2}\right)q^0\left(1+\|f\|_{L^2_q}\right)\cr
	&\le C(1+q^0)\|f\|_{L^2_q}.
\end{align*}	
For $\alpha\neq0$, we observe that
\begin{align*}
\partial^\alpha \Phi&= \sum_{|k|\le|\alpha|}\partial^k\left\{\frac{1}{n} \right\}\partial^{\alpha-k}\left\{\int_{\mathbb{R}^3}q^\mu F\,\frac{dq}{q^0}q_\mu\right\}\cr
&=\partial^\alpha\left\{\frac{1}{n} \right\}\left\{q^0\left(e_0+\int_{\mathbb{R}^3}f\sqrt{J^0}\,dq\right)+q\cdot \int_{\mathbb{R}^3}qf\sqrt{J^0}\,\frac{dq}{q^0}\right\}\cr
&+ \sum_{|k|<|\alpha|}\partial^k\left\{\frac{1}{n} \right\}\int_{\mathbb{R}^3}q^\mu \partial^{\alpha-k}f\sqrt{J^0}\,\frac{dq}{q^0}q_\mu,
 \end{align*}
yielding
\begin{align}\label{partial phi}
	\begin{split}
\left|		\partial^\alpha \Phi\right|&\le C  q^0\left|\partial^\alpha\left\{\frac{1}{n} \right\}\right|\left( 1+\|f\|_{L^2_q}\right)+ Cq^0\sum_{|k|<|\alpha|}\left|\partial^k\left\{\frac{1}{n} \right\}\right|\|\partial^{\alpha-k}f\|_{L^2_q}.
\end{split}\end{align}
To estimate $1/n$, we recall \eqref{n=} and \eqref{route pi} to write
\begin{align*}
	\frac{1}{n}&=\frac{1}{\sqrt{1+\Psi}}
\end{align*}
so that
	\begin{align*}
		\partial^k\left\{\frac{1}{\sqrt{1+\Psi}}\right\}&=\frac{\mathbb{P}(\Psi,\partial \Psi,\cdots,\partial^k \Psi)}{\left(1+\Psi\right)^{2^{|\alpha|}-\frac{1}{2}}}
	\end{align*}
for some homogeneous generic polynomial $\mathbb{P}$. Then an explicit computation using \eqref{Psi} gives
\begin{equation}\label{1/1+pi}
		\left|\partial^k\left\{\frac{1}{\sqrt{1+\Psi}}\right\}\right|\le C\sum_{|l|\le |k|}\|\partial^l f\|_{L^2_q}.
	\end{equation}
Substituting \eqref{1/1+pi} into \eqref{partial phi} leads to
\begin{align}\label{leads to}
	|\partial^\alpha \Phi|\le Cq^0\sum_{|k|\le |\alpha|}\|\partial^k f\|_{L^2_q}.
\end{align}
\newline
$\bullet$ Estimates of $\Phi_1$: In the same manner with $\Phi$, we can obtain
\begin{align}\label{obtain}
		|\partial^\alpha \Phi_1|\le C\sqrt{E(t)}\sum_{|k|\le |\alpha|}\|\partial^k f\|_{L^2_q}.
\end{align}
Now, (1) follows from  \eqref{Psi}, \eqref{leads to}, and combining \eqref{Psi1} and \eqref{obtain} gives (2).	
\end{proof}

The following two lemmas give the desired estimates for macroscopic fields.
\begin{lemma}\label{lem2} 	Suppose $E(t)$ is sufficiently small. Then we have
\begin{enumerate}
	\item $	\displaystyle|\partial^\alpha \{n -1\}|+|\partial^\alpha u |+|\partial^\alpha \{e -e_0\}|\le C\sum_{|k|\le|\alpha|}\|\partial^k f\|_{L^2_q}.$
	\item $\displaystyle|\partial^{\alpha} \{u ^0-1\}|\le C\sqrt{E(t)}\sum_{|k|\le|\alpha|}\|\partial^k f\|_{L^2_q}.$
\end{enumerate}

\end{lemma}
\begin{proof}
		$\bullet$ $\partial^\alpha \{n -1\}$ : It follows from \eqref{n=}, \eqref{route pi} and \eqref{H} that
	\begin{align*}
		\left|n -1\right| &= \left|\frac{\Psi}{2}-\frac{\Psi^2}{2(2+\Psi+2\sqrt{1+\Psi})}\right|\cr
		&\le C\|f\|_{L^2_q}+C\frac{\|f\|_{L^2_q}^2}{2(2-\|f\|_{L^2_q}+2\sqrt{1-\|f\|_{L^2_q}})}\cr
		& \le C\|f\|_{L^2_q}.
	\end{align*}
	For $\alpha\neq 0$, we write
	\begin{align*}
		\begin{split}
	\partial^\alpha \{n -1\}&=\partial^\alpha\left\{\sqrt{1+\Psi}\right\}=\frac{\mathbb{P}( \Psi,\partial \Psi,\cdots \partial^\alpha \Psi)}{\left(\sqrt{1+\Psi}\right)^{2^{|\alpha|}-1}}
\end{split}
\end{align*}
for some homogeneous generic polynomial $\mathbb{P}$.
Then, it is straightforward from \eqref{H} that
	\begin{align*}
\left|	\partial^\alpha \{n -1\}\right| 	&\le C\sum_{|k|\le |\alpha|}\|\partial^kf\|_{L^2_q}.
\end{align*}
	\noindent
$\bullet$ $\partial^\alpha u $ : The case of $\alpha=0$ follows similarly as in the above case:
	\begin{align*}
		|u |&=\left|\frac{1}{\sqrt{1+\Psi}}\int_{\mathbb{R}^3}q F\,\frac{dq}{q^0}\right|\leq  \frac{C\|  f\|_{L^2_q}}{\sqrt{1-\|f\|_{L^2_q}}}\leq C\|  f\|_{L^2_q}.
	\end{align*}
For the case $\alpha\neq 0$, we compute using \eqref{1/1+pi} as
\begin{align*}
|\partial^\alpha u |&= \left| \sum_{|k|\le|\alpha|}\partial^k\left\{\frac{1}{\sqrt{1+\Psi}} \right\}\int_{\mathbb{R}^3}\partial^{\alpha-k}\{F\}q\frac{dq}{q^0}  \right|\cr
&\le \sum_{|k|\le|\alpha|}\sum_{|l|\le |k|}\|\partial^l f\|_{L^2_q}\left| \int_{\mathbb{R}^3}\partial^{\alpha-k}\{f\}q\sqrt{J^0}\frac{dq}{q^0}  \right|\cr
&\le C \sum_{|k|\le |\alpha|}\sum_{|l|\le |k|}\|\partial^l f\|_{L^2_q}\|\partial^{\alpha-k}f\|_{L^2_q}.
\end{align*}
Applying $H^2(\mathbb{T}^{3})\subseteq L^{\infty}(\mathbb{T}^{3})$ to the lower order terms gives
\begin{align*}\
|\partial^\alpha u |&\le C\sum_{|k|\le|\alpha|}\|\partial^kf\|_{L^2_q}.
\end{align*}
\newline
	$\bullet$ $\partial^\alpha \{e -e_0\}$ : We observe from \eqref{e-e0} that $e-e_0$ takes the form of
	$$
		e-e_0= \int_{\mathbb{R}^3} \left(q^0-e_0\right)f\sqrt{J^0}\,dq -\int_{\mathbb{R}^3} f\sqrt{J^0}\,dq \int_{\mathbb{R}^3} q^0f\sqrt{J^0}\,dq+R_{I_2}(f) .
	$$
We notice that $R_{I_2}(f)$ consists entirely of the integrals for $f\sqrt{J^0}$, and can be estimated similarly as in the previous case. we omit the proof.
\noindent\newline
$\bullet$ $\partial^\alpha \{u^0 -1\}$ : We use \eqref{route pi} to express $u^0-1$ by
\begin{align*}
	u ^0-1&=\sqrt{1+|u |^2}-1\cr
	&=\frac{|u|^2}{2}-\frac{|u|^4}{2(2+|u|^2+2\sqrt{1+|u|^2})}.
\end{align*}
Then the previous result of $u$ gives
\begin{align*}
	|u ^0-1|&=\left|\frac{|u|^2}{2}-\frac{|u|^4}{2(2+|u|^2+2\sqrt{1+|u|^2})}\right|\cr
	&\le C\|f\|^2_{L^2_q}+\frac{C\|f\|_{L^2_q}^4}{8}\cr
	&\le C\sqrt{E(t)}\|f\|_{L^2_q}.
\end{align*}
Similarly, we have
\begin{align*}
\left|	\partial^\alpha \{u^0 -1\}\right|&=\left|\partial^\alpha \left\{\frac{|u |^2}{2}-\frac{|u |^4}{2(2+|u |^2+2\sqrt{1+|u|^2})} \right\}\right|\cr
	&\le C\sum_{|k|\le|\alpha|}|\partial^ku| |\partial^{\alpha-k}u| +\left|\frac{\mathbb{P}(\sqrt{1+|u ^2|},u ,\partial u ,\cdots, \partial^\alpha u )}{(2+|u |^2+2\sqrt{1+|u |^2})^{2^{|\alpha|}}(\sqrt{1+|u ^2|})^{2^{|\alpha|}-1}}\right|\cr
	&\le C\sqrt{E(t)}\sum_{|k|\le|\alpha|}\|\partial^k f\|_{L^2_q}.
\end{align*}

\noindent

\end{proof}
\begin{lemma}\label{U<}
Suppose $E(t)$ is sufficiently small. Then we have
\begin{enumerate}
	\item $\displaystyle|\partial^{\alpha} U|\le C\sum_{|k|\le|\alpha|}\|\partial^k f\|_{L^2_q}.$
		\item $\displaystyle|\partial^{\alpha}\{ U ^0-1\}|\le C\sqrt{E(t)}\sum_{|k|\le|\alpha|}\|\partial^k f\|_{L^2_q}.$
\end{enumerate}
		
\end{lemma}
\begin{proof}
$\bullet$ $\partial^\alpha U $: From the Cauchy-Schwartz inequality, we have
		\begin{align}\label{<uq<}
		\begin{split}
			1\le u^\mu q_\mu\le 2\sqrt{1+|u|^2}q^0.
		\end{split}
	\end{align}
Using this, we get
	\begin{align*}
|U|&=\left|u+3\frac{\int_{\mathbb{R}^3} q(u^\mu q_\mu)F\frac{dq}{q^0}-u\int_{\mathbb{R}^3} (u^\mu q_\mu)^2F\frac{dq}{q^0}}{4\int_{\mathbb{R}^3}(u^\mu q_\mu)^2F\frac{dq}{q^0}-\int_{\mathbb{R}^3}F\frac{dq}{q^0}}\right|\cr
&\le\left|u\right|+\frac{2\sqrt{1+|u|^2}\left|\int_{\mathbb{R}^3} qF\,dq\right|+4(1+|u|^2)\left|u\right|\int_{\mathbb{R}^3} q^0F\,dq}{\int_{\mathbb{R}^3}F\frac{dq}{q^0}}\cr
&=\left|u\right|+\frac{2\sqrt{1+|u|^2}\left|\int_{\mathbb{R}^3} qf\sqrt{J^0}\,dq\right|+4(1+|u|^2)\left|u\right|\left(e_0+\int_{\mathbb{R}^3} q^0f\sqrt{J^0}\,dq\right)}{\frac{K_1(\beta_0)}{K_2(\beta_0)}+\int_{\mathbb{R}^3}f\sqrt{J^0}\,\frac{dq}{q^0}}.
\end{align*}
We then apply Lemma \ref{lem2} to get
	\begin{align*}
 |U|&\le C\|f\|_{L^2_q}+C\frac{\big(1+\|f\|_{L^2_q}\big)\|f\|_{L^2_q}+\big(1+\|f\|_{L^2_q}\big)^3\|f\|_{L^2_q}}{\frac{K_1(\beta_0)}{K_2(\beta_0)}-\|f\|_{L^2_q}}\le C \|f\|_{L^2_q}.
\end{align*}
For the case of $\alpha\neq0$, we recall from \eqref{heat2} that
$$
\mathbf{q}=\int_{\mathbb{R}^3}  q f\sqrt{J^0}dq-\widetilde{h}(\beta_0)\int_{\mathbb{R}^3} qf\sqrt{J^0}\frac{dq}{q^0}+\Gamma_{3}^*(f)
$$
where $\Gamma_3^*(f)$ denotes
\begin{align*}
	\Gamma_{3}^*(f)	&=-\sum_{i=1}^3\int_{\mathbb{R}^3} q^if\sqrt{J^0}\frac{dq}{q^0}\int_{\mathbb{R}^3} q^i qf\sqrt{J^0}\frac{dq}{q^0}+\int_{\mathbb{R}^3} \Phi_1 qF\frac{dq}{q^0}\cr
	&  +e_0\int_{\mathbb{R}^3} q f\sqrt{J^0}\frac{dq}{q^0}\left\{\frac{\Psi}{2}+\frac{\Psi^3-3\Psi^2}{2(2+\Psi-\Psi^2+2\sqrt{1+\Psi})}\right\}\cr
	&  -\frac{1}{\sqrt{1+\Psi}}\left(\int_{\mathbb{R}^3} q^0 f\sqrt{J^0} dq+\int_{\mathbb{R}^3}\left(2q^0\Phi+\Phi^2 \right)F\frac{dq}{q^0} \right)\int_{\mathbb{R}^3} q f\sqrt{J^0}\frac{dq}{q^0}.
\end{align*}
Then, an explicit computation with Lemma \ref{lem22} and \eqref{1/1+pi} gives
$$
|\partial^\alpha \Gamma_{3}^*(f) |\le C\sum_{|k|\le |\alpha|}\|\partial^k f\|_{L^2_q}.
$$
Therefore,
\begin{equation}\label{q}
|\partial^\alpha \mathbf{q}|\le C\sum_{|k|\le |\alpha|}\|\partial^k f\|_{L^2_q}.
\end{equation}
On the other hand, we recall from \eqref{nh0} that
\begin{equation*}
nh=\widetilde{h}(\beta_0)+\frac{4}{3}\int_{\mathbb{R}^3} \left(2q^0\Phi+\Phi^2\right)J^0\frac{dq}{q^0}+\frac{4}{3}\int_{\mathbb{R}^3} ( u_{\alpha} q^\alpha)^2 f\sqrt{J^0}\frac{dq}{q^0}-\frac{1}{3}\int_{\mathbb{R}^3} f\sqrt{J^0}\frac{dq}{q^0},
\end{equation*}
which, together with Lemma \ref{lem22} and \eqref{<uq<}, leads to
\begin{equation}\label{nh1}
\left|nh-\widetilde{h}(\beta_0)\right|+ \left|\partial^\alpha \{nh\}\right|\le C\sum_{|k|\le |\alpha|}\|\partial^k f\|_{L^2_q}.
\end{equation}
Using this, we have
\begin{align}\label{q2}
	\begin{split}
	\left|\partial^{\alpha}\left\{\frac{1}{nh} \right\}\right|&=\left|\frac{\mathbb{P}\left(nh,\partial\{nh\},\cdots,\partial^{\alpha}\{nh\}\right)}{(nh)^{2^{|\alpha|}}}\right|\cr
	&\le C\sum_{|k|\le |\alpha|}\frac{\|\partial^kf\|_{L^2_q} }{\left(\widetilde{h}(\beta_0)-C\|f\|_{L^2_q}\right)^{2^{|\alpha|}}}\cr
	&\le  C\sum_{|k|\le |\alpha|} \|\partial^kf\|_{L^2_q}.
\end{split}\end{align}
 Combining \eqref{q} and \eqref{q2} gives
\begin{equation}\label{q<}
\left|\partial^\alpha \left\{\frac{\mathbf{q}}{nh}\right\}\right|\le C \sum_{|k|\le|\alpha|}\|\partial^k f\|_{L^2_q}.
\end{equation}
Therefore we have from Lemma \ref{lem2} and \eqref{q<} that
	\begin{align*}
|\partial^\alpha U|&=\left|\partial^\alpha u+\partial^\alpha \left\{\frac{\mathbf{q}}{nh}\right\}\right|\le C \sum_{|k|\le|\alpha|}\|\partial^k f\|_{L^2_q}.
\end{align*}
\newline
$\bullet$ $\partial^\alpha \{U^0-1\}$ : For $\alpha=0$, we observe that
\begin{align*}
	\mathbf{q}^0&=\int_{\mathbb{R}^3} (u^\mu q_\mu)Fdq-u^0\int_{\mathbb{R}^3} (u^\mu q_\mu)^2F\frac{dq}{q^0}\cr
	&=u^0\int_{\mathbb{R}^3}q^0Fdq-u\cdot \int_{\mathbb{R}^3}qFdq-u^0\int_{\mathbb{R}^3}\left(u^0q^0-u\cdot q\right)^2F\frac{dq}{q^0}.
\end{align*}	
Inserting $F=J^0+f\sqrt{J^0}$, a direct computation gives
\begin{align*}
	\mathbf{q}^0&=u^0\left(e_0+\int_{\mathbb{R}^3}q^0 f\sqrt{J^0}dq\right)-u\cdot \int_{\mathbb{R}^3}qf\sqrt{J^0}dq\cr
	&-(u^0)^3\left(e_0+\int_{\mathbb{R}^3}q^0f\sqrt{J^0}dq\right)+2(u^0)^2u\cdot\int_{\mathbb{R}^3}qf\sqrt{J^0}dq-u^0\int_{\mathbb{R}^3}(u\cdot q)^2F\frac{dq}{q^0}\cr
	&=-e_0u^0|u|^2+u^0\int_{\mathbb{R}^3}q^0 f\sqrt{J^0}dq-u\cdot \int_{\mathbb{R}^3}qf\sqrt{J^0}dq\cr
	&-(u^0)^3\int_{\mathbb{R}^3}q^0f\sqrt{J^0}dq+2(u^0)^2u\cdot\int_{\mathbb{R}^3}qf\sqrt{J^0}dq-u^0\int_{\mathbb{R}^3}(u\cdot q)^2F\frac{dq}{q^0},
\end{align*}
which, together with Lemma \ref{lem2}, gives
 \begin{align}\label{q0}
\begin{split}
	 	|\mathbf{q}^0|&\le C\left(\|f\|_{L^2_q}^2+\|f\|_{L^2_q}^3+\|f\|_{L^2_q}^4\right)
 	\le C\sqrt{E(t)}\|f\|_{L^2_q}.
\end{split} \end{align}
Combining Lemma \ref{lem2}, \eqref{nh1} and \eqref{q0}, we conclude that
	\begin{align*}
	|U^0-1|&=\left|u^0-1+\frac{\mathbf{q}^0}{nh}\right|\le C \sqrt{E(t)}\|f\|_{L^2_q}.
	\end{align*}
We omit the proof for $\alpha\neq 0$ to avoid the repetition.
\end{proof}

\subsection{Estimates of nonlinear perturbation $\Gamma$} We now estimate the nonlinear term. We first need to clarify the explicit form of the nonlinear perturbation $\Gamma_5(f)$.
 \begin{lemma}\label{Q} We have
	$$
	D_{n_{\theta},U_\theta^0,U_{\theta},e_{\theta}}^{2}J(\theta)=\mathcal{Q}J(\theta),
	$$
where $\mathcal{Q}$ is a $6\times6$ matrix whose elements are given by
\begin{align*}
&\mathcal{Q}_{1,1}=0,\quad \mathcal{Q}_{1,2}=-\frac{\beta_\theta}{n_{\theta}}q^0,\quad  	\mathcal{Q}_{1,i+2}=\frac{\beta_\theta}{n_{\theta}}q^i,\quad \mathcal{Q}_{1,6}=-\frac{1}{n_{\theta}\left\{\widetilde{e}\right\}^{\prime}(\beta_\theta)}\biggl(\frac{M^{\prime}(\beta_\theta)}{M(\beta_\theta)}+U_{\theta}^{\mu}q_{\mu}\biggl),  \cr
&\mathcal{Q}_{2,2}= \beta_\theta^2 (q^0)^2, \quad \mathcal{Q}_{2,i+2}=-\beta_\theta^2q^0q^i,\quad \mathcal{Q}_{2,6}=-\frac{1}{\left\{\widetilde{e}\right\}^{\prime}(\beta_\theta)} q^0+\frac{(\widetilde{e})^{-1}(e_\theta)}{\left\{\widetilde{e}\right\}^{\prime}(\beta_\theta)} q^0\biggl(\frac{M^{\prime}(\beta_\theta)}{M(\beta_\theta)}+U_{\theta}^{\mu}q_{\mu}\biggl),\cr
&\mathcal{Q}_{i+2,j+2}=\beta_\theta^2q^iq^j,\quad \mathcal{Q}_{i+2,6}=\frac{1}{\left\{\widetilde{e}\right\}^{\prime}(\beta_\theta)} q^i-\frac{(\widetilde{e})^{-1}(e_\theta)}{\left\{\widetilde{e}\right\}^{\prime}(\beta_\theta)} q^i\biggl(\frac{M^{\prime}(\beta_\theta)}{M(\beta_\theta)}+U_{\theta}^{\mu}q_{\mu}\biggl),\cr &\mathcal{Q}_{6,6}=-\left\{(\widetilde{e})^{-1}\right\}^{\prime\prime}(e_\theta)\biggl(\frac{M^{\prime}(\beta_\theta)}{M(\beta_\theta)}+U_{\theta}^{\mu}q_{\mu}\biggl)+\frac{1}{\left(\left\{\widetilde{e}\right\}^{\prime}(\beta_\theta)\right)^2}\biggl(\frac{M^{\prime}(\beta_\theta)}{M(\beta_\theta)}+ U_{\theta}^{\mu}q_{\mu}\biggl)^{2}\cr
		&\hspace{1cm} -\frac{1}{\left(\left\{\widetilde{e}\right\}^{\prime}(\beta_\theta)\right)^2}\biggl(\frac{M^{\prime\prime}(\beta_\theta)M(\beta_\theta)-\left\{M^{\prime}(\beta_\theta)\right\}^{2}}{M^{2}(\beta_\theta)}\biggl),
	\end{align*}
for $i,j=1,2,3$.
\end{lemma}
\begin{proof}
The proof is straightforward. We omit it.
\end{proof}
\begin{lemma}\label{nonlinear1}	Suppose $E(t)$ is sufficiently small. Then we have
$$
	\left|\int_{\mathbb{R}^3}\partial^\alpha_\beta\Gamma (f) g \,dq\right|\le C\sqrt{E(t)}\sum_{|k|\le |\alpha|}\|\partial^k f\|_{L^2_q}\|g\|_{L^2_q}.
	$$
\end{lemma}
\begin{proof}
We recall from Proposition \ref{lin3} that
		$$
		\Gamma (f)=\frac{U_\mu q^\mu}{q^0} \sum_{i=1}^5\Gamma_i (f)+\frac{1}{q^0}\left(\Phi+ \frac{\mathbf{q}_\mu q^\mu}{nh} \right) L(f).
	$$
For brevity, and to avoid unnecessary repetition, we only deal with
	$$
	I_1:=\frac{U_\mu q^\mu}{q^0} \Gamma_{3} (f), \qquad	I_2:=\frac{U_\mu q^\mu}{q^0} \Gamma_{5} (f).
	$$
Throughout the proof, we use the following elementary estimates without explicitly mentioning them:
$$
|\partial_{\beta}\hat{q}|+\left|\partial_{\beta}\left\{\frac{1}{q^{0}}\right\}\right|\le C,\quad |\partial_{\beta}\sqrt{J^0}|+\big|\partial_{\beta}\big(q^{\mu}\sqrt{J^0}\big)\big|\le C\sqrt{J^0},\quad \left|\partial_{\beta}\left\{\frac{1}{\sqrt{J^0}}\right\}\right|\le C\frac{1}{\sqrt{J^0}}.
$$
\newline
$\bullet$ Estimate of $I_1$ :
It follows from Lemma \ref{U<} that
\begin{align}\label{I1}
	\begin{split}
\left|\partial^\alpha_\beta\left\{I_1\right\}\right|&=\left|\partial^\alpha_\beta\left\{\frac{U_\mu q^\mu}{q^0} \Gamma_{3} (f)\right\}\right|\cr
&\le \sum_{\substack{|k|\le |\alpha|\cr |l|\le|\beta|}}\left|\biggl(\partial^{\alpha-k} \{U^0\}-\partial^{\alpha-k}\{U\}\cdot \partial_{\beta-l}\{\hat{q}\}\biggl) \partial^k_l\Gamma_{3} (f)\right| \cr
&\le  C\sum_{\substack{|k|\le |\alpha|\cr |l|\le|\beta|}}\biggl(1+\sum_{|m|\le|\alpha-k|}\|\partial^{m} f\|_{L^2_q}\biggl)\left|\partial^k_l\Gamma_{3} (f)\right|.
\end{split}\end{align}
Here $\Gamma_3(f)$ is given in Lemma \ref{lin2} that
\begin{align*}
	\Gamma_{3}(f)&=-\beta_0\left\{\frac{\Psi}{2}+\frac{\Psi^3-3\Psi^2}{2(2+\Psi-\Psi^2+2\sqrt{1+\Psi})}\right\}\int_{\mathbb{R}^3} q f\sqrt{J^0}\frac{dq}{q^0}\cdot q\sqrt{J^0}\cr
	&-\beta_0\left\{\frac{4\int_{\mathbb{R}^3} \left(2q^0\Phi+\Phi^2\right)J^0\frac{dq}{q^0}+4\int_{\mathbb{R}^3} (u_\mu q^\mu)^2 f\sqrt{J^0}\frac{dq}{q^0}-\int_{\mathbb{R}^3} f\sqrt{J^0}\frac{dq}{q^0}}{3\widetilde{h}(\beta_0)(n h ) } \right\}\cr
	&\quad\times \left(\int_{\mathbb{R}^3} (u_\mu q^\mu)q F\frac{dq}{q^0}-u \int_{\mathbb{R}^3} (u_\mu q^\mu)^2 F\frac{dq}{q^0}\right)\cdot q\sqrt{J^0}+\frac{\beta_0}{\widetilde{h}(\beta_0)}\Gamma_{3}^*(f)\cdot q\sqrt{J^0}
\end{align*}
where $\Gamma_{3}^*(f)$ denotes
\begin{align*}
	\Gamma_{3}^*(f)	&=-\sum_{i=1}^3\int_{\mathbb{R}^3} q^if\sqrt{J^0}\frac{dq}{q^0}\int_{\mathbb{R}^3} q^i qf\sqrt{J^0}\frac{dq}{q^0}+\int_{\mathbb{R}^3} \Phi_1 qF\frac{dq}{q^0}\cr
	&  +e_0\int_{\mathbb{R}^3} q f\sqrt{J^0}\frac{dq}{q^0}\left\{\frac{\Psi}{2}+\frac{\Psi^3-3\Psi^2}{2(2+\Psi-\Psi^2+2\sqrt{1+\Psi})}\right\}\cr
	&  -\frac{1}{\sqrt{1+\Psi}}\left(\int_{\mathbb{R}^3} q^0 f\sqrt{J^0} dq+\int_{\mathbb{R}^3}\left(2q^0\Phi+\Phi^2 \right)F\frac{dq}{q^0} \right)\int_{\mathbb{R}^3} q f\sqrt{J^0}\frac{dq}{q^0}.
\end{align*}
Using Lemma \ref{lem22}, \eqref{1/1+pi} and \eqref{nh1}, one can derive
\begin{align}\label{gamma*}
	\left|\partial^k\Gamma_{3}(f)\right|&\le C\sqrt{E(t)}\sum_{|l|\le |k|}\|\partial^l f\|_{L^2_q} \sqrt{J^0}.
\end{align}
We go back to \eqref{I1} with \eqref{gamma*} to have
 \begin{align*}
 	\left|\int_{\mathbb{R}^3} \partial^{\alpha}_{\beta}\{I_1\}g\, dq\right|&\le C\sum_{\substack{|k|\le |\alpha|\cr |l|\le|\beta|}}\biggl(1+\sum_{|m|\le|\alpha-k|}\|\partial^{m} f\|_{L^2_q}\biggl)\int_{\mathbb{R}^3} \left|\partial^k_l\left\{\Gamma_{3} (f)\right\} g \right|dq \cr
 	&\le  C\sqrt{E(t)}\sum_{|k|\le |\alpha|}\biggl(1+\sum_{|m|\le|\alpha-k|}\|\partial^{m} f\|_{L^2_q}\biggl)\sum_{|l|\le |k|}\|\partial^l f\|_{L^2_q}\int_{\mathbb{R}^3}   |g|\sqrt{J^0}  dq \cr
&\le C\sqrt{E(t)}\sum_{|k|\le|\alpha|}\|\partial^{k}f\|_{L^2_q}\|g\|_{L^2_q}.
 \end{align*}
In the last line, we applied $H^2(\mathbb{T}^{3})\subseteq L^{\infty}(\mathbb{T}^{3})$ to lower order terms.
\noindent \newline
$\bullet$ Estimate of $I_2$ :  Let $(n-1,U^0-1,U,e-e_{0})=(y_{1},y_{2},y_{3},y_{4},y_{5},y_6)$ and recall from Lemma \ref{Q} that
$$
D_{n_{\theta},U_\theta^0,U_{\theta},\alpha_{\theta}}^{2}J(\theta)=\mathcal{Q}J(\theta).
$$
We then have
$$
(y_{1},\cdots,y_{6})D_{n_{\theta},U_\theta^0,U_{\theta},e_{\theta}}^{2}J(\theta)(y_{1},\cdots,y_{6})^{T}=\sum_{i,j=1}^{6}y_{i}y_{j}\mathcal{Q}_{ij}J(\theta),
$$
which gives
\begin{align}\label{long} \partial^{\alpha}_{\beta}\Gamma_{5}(f)=\sum_{i,j=1}^6\sum_{\substack{|\alpha_{1}|+|\alpha_{2}|+|\alpha_{3}|\cr+|\alpha_{4}|=|\alpha|}}\!\!\!\!\partial^{\alpha_{1}}y_{i}\partial^{\alpha_{2}}y_{j}
\bigg\{\sum_{\substack{|\beta_{1}|+|\beta_{2}|+|\beta_{3}|\cr=|\beta|}}
	\int^1_0\partial^{\alpha_{3}}_{\beta_{1}}\mathcal{Q}_{ij}\partial^{\alpha_{4}}_{\beta_{2}}J(\theta)d\theta \partial_{\beta_{3}}\bigg(\frac{1}{\sqrt{J^0}}\bigg)\bigg\}.
\end{align}
\noindent  {\bf Claim:} Assume $E(f)(t)$ is sufficiently small. Then, there exist   positive constants $C$, $C^{\prime}>0$, independent of $\theta$, such that
$$J(\theta)\le Ce^{-C^{\prime}\sqrt{1+|q|^{2}}}$$
and
\begin{align*}
	|\partial^{\alpha}_{\beta}\mathcal{Q}_{ij}|&\le C\sum_{|k|\le|\alpha|}\|\partial^{k}f\|_{L^2_q},\cr
	|\partial^{\alpha}_{\beta}J(\theta)|&\le C\sum_{|k|\le|\alpha|}\|\partial^{k}f\|_{L^2_q}e^{-C^{\prime}\sqrt{1+|q|^{2}}}.
\end{align*}
\noindent {\bf Proof of the claim:}
We observe from Lemma \ref{beta decreasing} that since $\widetilde{e}(\beta)$ is a decreasing function, its inverse $\beta_{\theta}=(\widetilde{e})^{-1}(e_\theta)$ is also decreasing. So it follows from Lemma \ref{lem2} that $\beta_{\theta}$ is bounded by
$$
(\widetilde{e})^{-1}\left( e_0+\sqrt{E(t)}\right)\le \beta_\theta \le (\widetilde{e})^{-1}\left( e_0-\sqrt{E(t)}\right).
$$
From this observation with Lemma \ref{U<}, we have
\begin{align*}
\beta_{\theta}U_{\theta}^{\mu}q_{\mu}&\ge (\widetilde{e})^{-1}\left( e_0+\sqrt{E(t)}\right)\left( \min\{1,U^0\} \sqrt{1+|q|^{2}}-\max_{U}\{|U||q|\} \right)\cr
	&= (\widetilde{e})^{-1}\left( e_0+\sqrt{E(t)}\right)\left\{\big(1-\sqrt{E(t)}\big)\sqrt{1+|q|^2}-\big(\sqrt{E(t)}\big)|q|\right\}\cr
	&\ge (\widetilde{e})^{-1}\left( e_0+\sqrt{E(t)}\right)\big(1-2\sqrt{E(t)}\big)\sqrt{1+|q|^2}.
\end{align*}
When $E(t)$ is sufficiently small, we can take $C^{\prime}$ such that
\begin{align}\label{beta0}
 (\widetilde{e})^{-1}\left( e_0+\sqrt{E(t)}\right)\big(1-2\sqrt{E(t)}\big)>C^{\prime}>\frac{3}{4}\beta_0.
\end{align}
 We thus have from Lemma \ref{lem2} that
\begin{align}\label{JJ}
J(\theta)=\frac{n_{\theta}}{M(\beta_{\theta})}e^{-\beta_{\theta}U_{\theta}^{\mu}q_{\mu}}\le Ce^{-C^{\prime}\sqrt{1+|q|^{2}}}.
\end{align}
This gives the first estimate.\newline
For the second estimate, we observe by an explicit, tedious computation that the derivatives of $J(\theta)$ and $\mathcal{Q}$ are expressed as
\begin{align*}
	\partial^{\alpha}_{\beta}J(\theta) &=\sum_{\substack{|\alpha_1|+\cdots+|\alpha_4|\cr =|\alpha|}}\sum_{l} \mathbb{P}_J( n_{\theta},\cdots,\partial^{\alpha_{1}} n_{\theta},U^\mu_{\theta},\cdots,\partial^{\alpha_{2}} U_{\theta},e_{\theta},\cdots,\partial^{\alpha_{3}} e_{\theta})_l\cr
	&\quad\times\biggl(\frac{\mathbb{Q}_J( q^{\mu},M(\beta_{\theta}),\cdots,\partial_{e_\theta}^{\alpha_4}M(\beta_{\theta}),\beta_{\theta},\cdots,\partial_{e_\theta}^{\alpha_4}\beta_{\theta}}{\mathbb{M}_J(n_{\theta},M(\beta_{\theta}),q^{0})}\biggl)_lJ(\theta),\cr
	\partial^{\alpha}_{\beta} \mathcal{Q}_{ij}&=\partial^{\alpha}_{\beta} \mathcal{Q}_{ij}(n_{\theta},U^\mu_{\theta},e_{\theta},q^\mu)\cr
	&=\biggl\{\sum_{\substack{|\alpha_1|+\cdots+|\alpha_4|\cr =|\alpha|}}\sum_{k}\mathbb{P}_{\mathcal{Q}}(n_{\theta},\cdots,\partial^{\alpha_{1}} n_{\theta},U^\mu_{\theta},\cdots,\partial^{\alpha_{2}} U_{\theta}^\mu,e_{\theta},\cdots,\partial^{\alpha_{3}} e_{\theta})_{k}\cr &\quad\times\biggl(\frac{\mathbb{Q}_{\mathcal{Q}}(q^{\mu},M(\beta_{\theta}),\cdots,\partial_{e_\theta}^{\alpha_4}M(\beta_{\theta}),\beta_{\theta},\cdots,\partial_{e_\theta}^{\alpha_4}\beta_{\theta}}{\mathbb{M}_{\mathcal{Q}}(n_{\theta},M(\beta_{\theta}), q^{0})}\biggl)_{k}\biggl\}_{ij},
\end{align*}
for some generically defined polynomials $\mathbb{P}$, $\mathbb{Q}$ and $\mathbb{M}$. Then, the desired result follows from Lemma \ref{lem2}, Lemma \ref{U<} and the estimate (\ref{JJ}).
This ends the proof of the claim.\newline

Now, substituting the estimates in the claim into (\ref{long}), we obtain
\begin{align*}
&\left|\partial^{\alpha}_{\beta}\Gamma_{5}(f)\right|\cr	
&\le \sum_{\substack{|\alpha_{1}|+|\alpha_{2}|+|\alpha_{3}|\cr+|\alpha_{4}|=|\alpha|}}
|\partial^{\alpha_{1}}y_{i}||\partial^{\alpha_{2}}y_{j}|\sum_{\substack{|\beta_{1}|+|\beta_{2}|+|\beta_{3}|\cr=|\beta|}}
	\int^1_0|\partial^{\alpha_{3}}_{\beta_{1}}\mathcal{Q}_{ij}||\partial^{\alpha_{4}}_{\beta_{2}}J(\theta)|d\theta  \bigg|\partial_{\beta_{3}}\left\{\frac{1}{\sqrt{J^0}}\right\}\!\bigg|\cr
	&\le C\!\!\!\sum_{\substack{|\alpha_{1}|+|\alpha_{2}|+|\alpha_{3}|\cr+|\alpha_{4}|=|\alpha|}}\sum_{|k|\le|\alpha_{1}|}\|\partial^{k}f\|_{L^2_q}\!\!\sum_{|k|\le|\alpha_{2}|}\|\partial^{k}f\|_{L^2_q}\!\!\sum_{|k|\le|\alpha_{3}|}\|\partial^{k}f\|_{L^2_q}\!\!\sum_{|k|\le|\alpha_{4}|}\|\partial^{k}f\|_{L^2_q}e^{-C^{\prime}\sqrt{1+|q|^2}}\frac{1}{\sqrt{J^0}}\cr
	&\equiv B(f)\frac{e^{-C^{\prime}\sqrt{1+|q|^2}}}{\sqrt{J^0}}.
\end{align*}
Then, we apply the Sobolev embedding $H^2(\mathbb{T}^3)\subseteq L^{\infty}(\mathbb{T}^3)$ to the lower order terms  to get
\begin{align*}
	B(f)\le C\sqrt{E(t)}\sum_{|k|\le|\alpha|}\|\partial^{k}f\|_{L^2_q},
\end{align*}
and observe from (\ref{beta0}) that $C^\prime-\frac{\beta_0}{2}>\frac{\beta_0}{4}$, that gives
\[
\frac{e^{-C^{\prime}\sqrt{1+|q|^2}}}{\sqrt{J_0}}\leq e^{-\frac{1}{4}\beta_0\sqrt{1+|q|^2}}.
\]
Therefore,
\begin{align*}
&\left|\partial^{\alpha}_{\beta}\Gamma_{5}(f)\right|\leq C\sqrt{E(t)}\sum_{|k|\le|\alpha|}\|\partial^{k}f\|_{L^2_q}e^{-\frac{1}{4}\beta_0\sqrt{1+|q|^2}}.
\end{align*}
The desired estimate follows directly from this:
\begin{align*}
	\left|\int_{\mathbb{R}^3} \partial^{\alpha}_{\beta}\{I_2\}g\, dq\right|&\le   \sum_{\substack{|k|\le |\alpha|\cr |l|\le|\beta|}}\int_{\mathbb{R}^3} \left|\biggl(\partial^{\alpha-k} \{U^0\}-\partial^{\alpha-k}\{U\}\cdot \partial_{\beta-l}\{\hat{q}\}\biggl) \partial^k_l\Gamma_{5} (f)g\right|\, dq\cr
	&\le C\sqrt{E(t)}\sum_{|k|\le|\alpha|}\|\partial^{k}f\|_{L^2_q}\biggl(1+\sum_{|m|\le|\alpha-k|}\|\partial^{m} f\|_{L^2_q}\biggl)\int_{\mathbb{R}^3} e^{-\frac{1}{4}\beta_0\sqrt{1+|q|^{2}}}|g|dq\cr
	&\le C\sqrt{E(t)}\sum_{|k|\le|\alpha|}\|\partial^{k}f\|_{L^2_q}\|g\|_{L^2_q}.
\end{align*}
\end{proof}
The following lemma on the difference of distribution functions is also needed for the local existence and uniqueness. We omit the proof since it can be treated similarly.
\begin{lemma}\label{nonlinear2}
	Assume $\bar{F}:=J^0+\bar{f}\sqrt{J^0}$ is another solution of \eqref{AWRBGK1}. For sufficiently small $E(f)(t)$ and $E(\bar{f})(t)$, we then have
	$$
	\left|\int_{\mathbb{T}^3}\int_{\mathbb{R}^3} \left\{\Gamma(f)-\Gamma(\bar{f})\right\}(f-\bar{f}) \,dqdx\right|\le
	C \|f-\bar{f}\|^{2}_{L^2_{x,q}}.$$
\end{lemma}

%
%
%
%
%
%
%
\section{Proof of the main result}
Now, we are ready to prove the main result. Since it is rather standard, we only sketch the proof.
\subsection{Local existence}
With the estimates on the nonlinear parts established in the previous section, the local in time existence then follows by standard argument \cite{Guo whole,Guo VMB}:
\begin{proposition}\label{thm8}
	Let $N\geq4$ and  $F_{0}=J^0+\sqrt{J^0}f_{0}\ge 0.$ 
	Then there exist $M_{0}>0, T_{*}>0$, such that if $T_{*}\le\frac{M_{0}}{2}$ and $E(f_{0})\le \frac{M_{0}}{2}$, there is a unique global solution $F(x,q,t)$ to the Anderson-Witting model \eqref{AWRBGK1} such that
	\begin{enumerate}
		\item[(1)] The energy functional is continuous in $[0,T_{*})$ and uniformly bounded:
		$$
		\sup_{0\le t\le T_{*}}E(f)(t)\le M_{0}.
		$$
		\item[(2)] The distribution function remains positive in $[0,T_{*})$:
		$$
		F(x,q,t)=J^0+\sqrt{J^0}f(x,q,t)\ge 0.
		$$
	\end{enumerate}
\end{proposition}
%
%
%
%
\subsection{Coercivity of L}
We decompose $f$ into the macro and micro parts:
$$
f=P(f)+\{I-P\}(f)
$$
where $P(f)$ takes the form of
\begin{align*}
P(f)=\tilde{a}\sqrt{J^0}+b\cdot q\sqrt{J^0}+cq^0\sqrt{J^0},
\end{align*}
and rewrite the linearized Anderson-Witting model \eqref{LAW} as follows:
\begin{align*}
	\{\partial_{t}+\hat{q}\cdot\nabla_{x}\}P(f)&=\{-\partial_{t}-\hat{q}\cdot\nabla_{x}+L\}\{I-P\}(f)+\Gamma (f)\cr
	&\equiv l\{I-P\}(f)+h(f).
\end{align*}
Comparing both sides of the equation with respect to the basis:
\begin{equation}\label{basis}
\{e_{a_0}, e_{a_i}, e_{bc_i}, e_{ij}, e_{c}\}
=\Big\{\sqrt{J_0},\ \frac{q_{i}}{q^{0}}\sqrt{J_0},\ q_{i}\sqrt{J_0},\ \frac{q_{i}q_{j}}{q^{0}}\sqrt{J_0},\ q^{0}\sqrt{J_0}\Big\},
\end{equation}
we obtain the following micro-macro system:
	\begin{enumerate}
		\item $\partial_{t}\tilde{a}=l_{a_0}+h_{a_0}$,
		\item $\partial_{t}c=l_{c}+h_{c}$,
		\item $\partial_{t}b_{i}+\partial_{x_{i}}c=l_{bc_i}+h_{bc_i}$,
		\item $\partial_{x_{i}}\tilde{a}=l_{ai}+h_{ai}$,
		\item $(1-\delta_{ij})\partial_{x_{i}}b_{j}+\partial_{x_{j}}b_{i}=l_{ij}+h_{ij}$,
	\end{enumerate}
where $l_{a_0},\cdots, l_{c}$ and $h_{a_0},\cdots,h_{c}$ denote the inner product of $l\{I-P\}(f)$ and $h(f)$ with the corresponding basis (\ref{basis}).
The standard analysis of the system \cite{Guo whole,Guo VMB} gives
\begin{align*}
\sum_{|\alpha|\le N}\|\partial^{\alpha}P(f)\|^{2}_{L^2_{x,q}}\le
		 C\sum_{|\alpha|\le N}\big\{\|\{I-P\}(\partial^{\alpha}f)\|^{2}_{L^2_{x,q}}+\sqrt{E(t)}\|\partial^{\alpha}f\|^{2}_{L^2_{x,q}}\big\}.
\end{align*}
This, combined with Proposition \ref{pro} (2), gives the dissipative estimate of $L$ for sufficiently small $E(f)$:
	\begin{equation}\label{coercivity}
	\sum_{|\alpha|\le N}\langle L(\partial^{\alpha}f),\partial^{\alpha}f\rangle_{x,q}\le-\delta\sum_{|\alpha|\le N}\|\partial^{\alpha}f\|^{2}_{L^2_{x,q}}
	\end{equation}
for some $\delta>0$.
\subsection{Proof of Theorem \ref{main2}}
Applying $\partial^{\alpha}_{\beta}$ to (\ref{LAW}), taking inner product with $\partial^{\alpha}_{\beta}f$, and employing Lemma \ref{nonlinear1} and \eqref{coercivity}, we have from standard arguments \cite{Guo VMB} that
\begin{align*}
	\frac{1}{2}\frac{d}{dt}\|\partial^{\alpha}f\|^{2}_{L^2_{x,q}}+\delta\|\partial^{\alpha}f\|^{2}_{L^2_{x,q}}&\le C\sqrt{E(t)}E(t)\quad (\beta=0),
\end{align*}
and

\begin{align*}
	&\frac{d}{dt}\|\partial^{\alpha}_{\beta}f\|^{2}_{L^2_{x,q}}+\delta\|\partial^{\alpha}_{\beta}f\|^{2}_{L^2_{x,q}}\le C\biggl\{\sum_{|k|<|\beta|}\sum_{i=1}^{3}\|\partial_{x_{i}}\partial^{\alpha}_{k}f\|^{2}_{L^2_{x,q}} +\sqrt{E (t)}E(t)\biggl\}\quad (\beta\neq0).
\end{align*}
We then combine these estimates to derive the following energy estimate \cite{Guo VMB}:
\begin{equation*}
	\sum_{|\alpha|+|\beta|\le N}\biggl(C_{|\beta|}\frac{d}{dt}\|\partial^{\alpha}_{\beta}f\|^{2}_{L^2_{x,q}}+\delta_{N}\|\partial^{\alpha}_{\beta}f\|^{2}_{L^2_{x,q}}\biggl)\le C_{N*}\sqrt{E(t)} E(t)
\end{equation*}
for some positive constants $C_{|\beta|}$, $\delta_{N}$.
Then, the desired result follows from the standard continuity argument.\newline\newline
\noindent{\bf Acknowledgement}
The work of Seok-Bae Yun was supported by Samsung Science and Technology Foundation under Project Number SSTF-BA1801-02.

\bibliographystyle{amsplain}

\end{document}